\documentclass[a4wider]{amsart}
\usepackage{enumerate}
\usepackage{float}
\usepackage{tikz , graphicx}
\usepackage{chngcntr}
\graphicspath{ {./images/} }
\usepackage{fontenc}
\usepackage{color}
\usepackage{amsfonts,amsmath,amssymb,amsthm}
\usepackage{dsfont}
\usepackage{latexsym}

\usepackage[T1]{fontenc}
\usepackage{amsmath,amsfonts,amstext}
\usepackage{url}

\usepackage[british]{babel}
\usepackage{paralist}               	% eingebaut UG
\usepackage[parfill]{parskip}			% eingebaut UG macht es uebersichtlicher

\newtheorem{theorem}{Theorem}[section]
\newtheorem{corollary}[theorem]{Corollary}

\newtheorem{definition}[theorem]{Definition}

\newtheorem{question}[theorem]{Question}
\newtheorem{lemma}[theorem]{Lemma}
\newtheorem{proposition}[theorem]{Proposition}

\newtheorem{remarka}[theorem]{Remark}

\newenvironment{psmallmatrix}
  {\left(\begin{smallmatrix}}
  {\end{smallmatrix}\right)}

\newtheorem*{T*}{Theorem}
\newtheorem*{A*}{Proposition}
\newtheorem*{Cor*}{Corollary}
\theoremstyle{definitionbreak}
\theoremstyle{definitionbreak}\newtheorem*{D*}{Definition}

\title{Action convergence of operators and graphs}
\author{\'Agnes Backhausz, Bal\'azs Szegedy}

\date\today
\keywords{graph limits, operator, random matrix}
\subjclass[2000]{05C50}

\begin{document}

\begin{abstract} We present a new approach to graph limit theory which unifies and generalizes the two most well developed directions, namely dense graph limits (even the more general $L^p$ limits) and Benjamini--Schramm limits (even in the stronger local-global setting). We illustrate by examples that this new framework provides a rich limit theory with natural limit objects for graphs of intermediate density. Moreover, it provides a limit theory for bounded operators (called $P$-operators) of the form $L^\infty(\Omega)\to L^1(\Omega)$ for probability spaces $\Omega$. We introduce a metric to compare $P$-operators (for example finite matrices) even if they act on different spaces. We prove a compactness result which implies that in appropriate norms, limits of uniformly bounded $P$-operators can again be represented by $P$-operators. We show that limits of operators representing graphs are self-adjoint, positivity-preserving $P$-operators called graphops. Graphons, $L^p$ graphons and graphings (known from graph limit theory) are special examples for graphops. We describe a new point of view on random matrix theory using our operator limit framework. %This direction will be continued in a follow up paper where we analyze the eigenvectors of nonsymmetric random matrices.
\end{abstract}

\maketitle

\section{Introduction}

A fundamental question posed in the emerging field of graph limit theory is the following: {\it How can we measure similarity of graphs?} Each branch of graph limit theory is based on a similarity metric \cite{lovaszbook}. Experience shows that, to be useful in applications, the similarity metric should satisfy a few natural properties. 
\begin{enumerate}
\item {\bf (Expressive power)}{\it The similarity metric should be fine enough to provide a rich enough picture of graph theory. }
\item {\bf (Compactness)}{\it The similarity metric should be coarse enough to provide many interesting Cauchy convergent graph sequences.}
\item {\bf (Limit objects)}{\it Limits of Cauchy convergent sequences of graphs should be naturally represented by "graph-like" analytic objects.}

\end{enumerate}
The tension between the first and the second requirement makes the search for useful similarity metrics especially interesting. The so-called dense graph limit theory is based on a set of equivalent metrics. One of them is the $\delta_\square$-distance \cite{BCLTV, LSz, LSz2}. Convergence in $\delta_\square$ is equivalent to the convergence of subgraph densities. The completion of the set of all graphs in this metric is compact, and thus every graph sequence has a convergent sub-sequence, which is a very useful property. A shortcoming of dense graph limit theory is that sparse graphs are considered to be similar to the empty graph and thus it has not enough expressive power to study graphs in which the number of edges is sub-quadratic in the number of vertices. Another similarity notion was introduced by Benjamini and Schramm \cite{BS} to study bounded degree graphs that are basically the sparsest graphs. This metric requires an absolute bound for the largest degree and hence it can not be used for graphs with super-linear number of edges. Graph sequences in which the number of edges is super-linear and sub-quadratic in terms of the number of vertices are called graphs of {\it intermediate density}. 

Finding useful similarity notions for graphs of intermediate density is a major research direction in graph limit theory. There are many promising non-equivalent approaches to this subject \cite{BCCH, BCCZ, BCCZ2, frenkel, janson, nesetrilkonyv, nesetril2, sparse}. 
However, none of them provides a real unification of the most well-developed branches: dense graph limit theory (together with its $L^p$ extension \cite{BCCZ, BCCZ2}), Benjamini--Schramm limit theory (together with the stronger local-global convergence, see e.g.\ \cite{BR, HLSz}) and corresponding limit objects: graphons, $L^p$ graphons and graphings. 

{\it In this paper we take a new point of view on the subject. Instead of considering graphs as static structures, we focus more on the action and dynamics generated by graphs. One can associate various operators with graphs. The most well-known examples are: adjacency operators, Laplace operators and Markov kernel operators (related to random walks). We formulate a framework theory of operator convergence and apply it to graph theory through representing operators.}

The dynamical aspect is present in many existing limit theories. However, it has not been exploited to unify them. Limit objects such as graphons and graphings act on $L^2$ spaces of probability spaces. (Even the so called $L^p$ graphons can be viewed as operators of the form $L^q(\Omega)\to L^p(\Omega)$, where $\Omega$ is a probability space.) While graphons are compact operators represented by measurable functions of the form $W:[0,1]^2\to [0,1]$, graphings are non-compact and are represented by singular measures on $[0,1]^2$ concentrated on edge sets of bounded degree Borel graphs \cite{EG, HLSz}. {\it A common property of all of these objects is that they are bounded operators in an appropriate norm and they act on function spaces of random variables.} Graphons and graphings are bounded in the usual $L^2$ operator norm $\|.\|_{2\to 2}$, and $L^p$ graphons are bounded in the $\|.\|_{q\to p}$ norm, where $p^{-1}+q^{-1}=1$.

In spite of the fact that existing convergence notions for graphons and graphings are intuitively similar, the exact connection has not yet been explained from a functional analytic point of view. In this paper we introduce a general convergence notion for operators acting on functions on probability spaces. We show that graphon convergence, $L^p$ graphon convergence and local-global convergence of graphings are all special cases of this general convergence notion. Moreover, we obtain a very general framework for graph limit theory by studying the convergence of operator representations of graphs. 

We also demonstrate that the new limit theory for operators has applications beyond graph theory through a new approach to random matrix theory.  
An important motivation for this paper comes from a previous result by the authors which proves Gaussianity for almost eigenvectors of random regular graphs using graph limit techniques (local-global limits) and information theory \cite{almost}. It is very natural to ask if similar limit techniques can be used to study dense random matrices such as matrices with i.i.d $\pm 1$ entries. Available graph limit techniques proved to be too weak for this problem. Dense random matrices (when regarded as weighted graphs) converge to trivial objects in dense graph limit theory. Note that an interesting connection between dense graph limits and random matrices was investigated in  \cite{male, zhu}.

We propose a new limit approach for matrices, graphs and operators which is based on the following quite simple and natural probabilistic view point on matrix actions.  Let $A\in\mathbb{R}^{n\times n}$ be an arbitrary matrix and let $v\in\mathbb{R}^n$ be a vector. Let $M$ denote the $2\times n$ matrix whose rows are $v$ and $vA$. Each column of $M$ is an element in $\mathbb{R}^2$, thus, by choosing a random column, we obtain a probability distribution $\mu_v$ on $\mathbb{R}^2$. The following interesting question arises:

\noindent{\it How much do we learn about $A$ if we know the set of all probability measures $\mu_v$ arising this way?}

It is easy to see for example that $A$ is the identity matrix if and only if $\mu_v$ is supported on the line $y=x$ in $\mathbb{R}^2$ for every $v\in\mathbb{R}^n$. The matrix $A$ is degenerate if and only if there is a measure $\mu_v$ which is not the Dirac measure $\delta_{(0,0)}$, but it is supported on the line $y=0$.

Philosophically, we regard each measure $\mu_v$ as an observation associated with the action of $A$ and we regard the set of all possible observations $\{\mu_v : v\in \mathbb{R}^n\}$ as the profile of $A$. A useful fact about  profiles is that they allow us to compare matrices of different sizes, since they are sets of probability measures on $\mathbb{R}^2$ independently of the sizes of the matrices. Another nice fact is that the profile of $A$ contains rather detailed information about the eigenvalues of $A$ and the entry distributions of the corresponding eigenvectors. It is easy to see that $v$ is an eigenvector with eigenvalue $\lambda$ if and only if the measure $\mu_v$ is supported on the line $y=\lambda x$ in $\mathbb{R}^2$. The entry distribution of $v$ is simply the distribution of the $x$ coordinates in $\mu_v$. 

%\counterwithout{figure}{section}
%\begin{figure}[htbp] \centering
   
%      \includegraphics[scale=0.37]{meas1.jpg} \hfill \includegraphics[scale=0.37]{meas1.jpg} 
 % \caption{A probability measure $\mu_v$ in the profile of a $2000\times 2000$ random matrix with Gaussian entries} \label{figcube}
 
%\end{figure}

\begin{figure}[!tbp]
  \centering
  \begin{minipage}[b]{0.4\textwidth}
    \includegraphics[width=\textwidth]{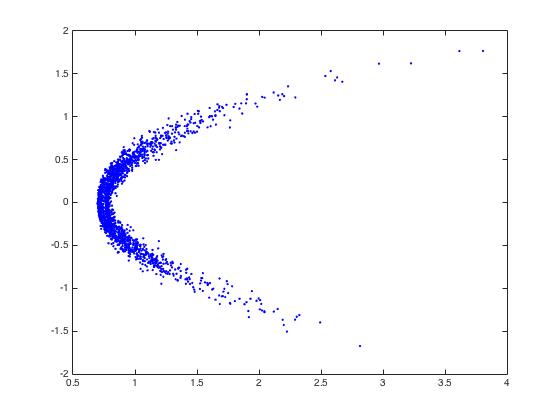}
    %\caption{Flower one.}
  \end{minipage}
  \hfill
  \begin{minipage}[b]{0.4\textwidth}
    \includegraphics[width=\textwidth]{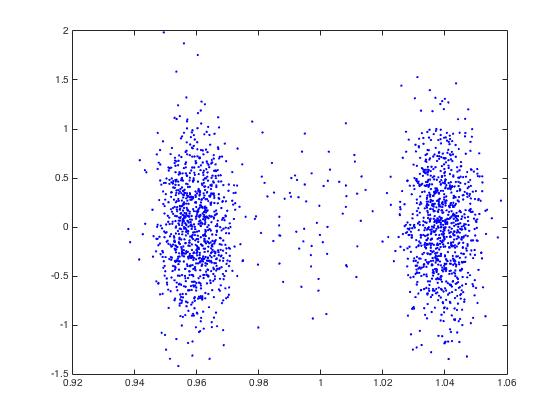}
    %\caption{Flower two.}
  \end{minipage}
  \caption{Two probability measures of the form $\mu_v$ in the profile of a $2000\times 2000$ random matrix.} \label{figmeas}
\end{figure}

It is useful to extend this idea to the case when $k$ vectors $v_1,v_2,\dots,v_k$ are considered simultaneously. (For some technical reasons we will assume that $v_1,v_2,\dots,v_k$ are in $[-1,1]^n$.) In this case $M$ is the $2k\times n$ matrix with rows $\{v_i\}_{i=1}^k$ and $\{v_iA\}_{i=1}^k$. A random column in $M$ yields a probability distribution on $\mathbb{R}^{2k}$, and the $k$-profile $\mathcal{S}_k(A)$ of $A$ is the collection of all such probability measures. We regard $A$ and $B$ to be similar if for small natural numbers $k$ their $k$-profiles are close in the Hausdorff metric $d_H$ defined for sets of probability measures on $\mathbb{R}^{2k}$ based on the L\'evy--Prokhorov metric for individual measures (precise definition will be given in Section \ref{chaplim}.) This similarity can be metrized by the formula
$$d_M(A,B):=\sum_{k=1}^\infty 2^{-k}d_H(S_k(A),S_k(B)).$$ A sequence of matrices converges in this metric if for every fixed $k$, their $k$-profiles converge in $d_H$.

The above ideas generalize naturally to the framework where $(\Omega,\mathcal{A},\mu)$ is a probability space and $A$ is an operator of the form $A:L^\infty(\Omega)\to L^1(\Omega)$. Such operators with an appropriate boundedness condition will be called {\bf $P$-operators} (Definition \ref{def:pop}). If $v\in L^\infty(\Omega)$, then both $v$ and $vA$ are random variables, and their joint distribution is a measure $\mu_v$ on $\mathbb{R}^2$. This allows us to define $k$-profiles, metric and convergence for $P$-operators similarly as we defined them for matrices. Note that matrices are special $P$-operators, where the probability space is $[n]:=\{1,2,\dots,n\}$ with the uniform distribution. In this case $L^\infty([n])=L^1([n])=\mathbb{R}^{[n]}$, and every matrix is a $P$-operator. Note that both graphons (symmetric measurable functions of the form $W:[0,1]^2\to[0,1]$) and graphings (certain bounded degree Borel graphs on measure spaces) are special $P$-operators. We prove the next surprising result.

\begin{theorem} $P$-operator convergence (given by Definition $\ref{limconv}$) restricted to the set of graphons is the same as graphon convergence. Furthermore, $P$-operator convergence restricted to the set of graphings is equivalent to the local-global convergence of graphings.
\end{theorem}

The proof of the above theorem relies on a recent result of the second author which reformulates local-global convergence in terms of colored star metric \cite{star}.  
Our main theorem (in an informal language) is the following.

\begin{theorem}{\bf (Compactness and limit object)} Every sequence of $P$-operators with uniformly bounded $\|.\|_{\infty\to 1}$ norms has a Cauchy convergent sub-sequence with respect to $d_M$. Furthermore, if $p,q\in [1,\infty)$, then every Cauchy convergent sequence of $\|.\|_{p\to q}$ uniformly bounded $P$-operators has a limit, which is also a $P$-operator, and the same bound applies for its norm.
\end{theorem}

We show that, under certain boundedness conditions, a number of important operator properties are closed with respect to $P$-operator convergence. This includes self-adjointness, positivity and the positivity-preserving property. A $P$-operator $A:L^\infty(\Omega)\to L^1(\Omega)$ is called positivity-preserving if $vA$ is a non-negative function on $\Omega$ whenever $v\in L^\infty(\Omega)$ is non-negative. The graph-like objects in the universe of $P$-operators are special $P$-operators called graphops.

\begin{definition} A {\bf graphop} is a positivity-preserving, self adjoint $P$-operator. 
\end{definition} 

A particularly nice property of graphops is that they can be represented by symmetric finite measures $\nu$ on $\Omega^2$ with absolutely continuous marginals (see Theorem \ref{measrepthm}.)

\noindent{\it Intuitively, the measure $\nu$ plays the role of the "edge set" of the graphop $A$. When scaled to a probability measure, $\nu$ can be used to sample a random element in $\Omega\times\Omega$, which is the analogue of a random directed edge in a finite graph. By disintegrating $\nu$, we obtain measures $\nu_x$ for every $x\in\Omega$ describing "neighborhoods" in $A$.} 

Adjacency matrices of graphs (or positive weighted graphs), graphons, $L^p$-graphons and graphings are all examples for graphops. A concrete example for a graphop (called spherical graphop), which is none of the previous classes, is explained on Figure \ref{sphericgraph}.

\begin{remarka} Graphops have {\it "edge densities"} and {\it "degrees"}. If $A:L^\infty(\Omega)\to L^1(\Omega)$ is a graphop, then $1_\Omega A$ is a non-negative function. The expected value of $1_\Omega A$ is the edge density of $A$. The value of $1_\Omega A$ at a point $x\in\Omega$ is the "degree" of $x$. The distribution of the random variable $1_\Omega A$ is the "degree distribution" of $A$.
\end{remarka}

\begin{figure}
\tikzstyle{every node}=[circle, draw, fill=black!50,
                        inner sep=0pt, minimum width=4pt]

\begin{tikzpicture}[thick,scale=0.4]%

\begin{scope}[shift={(-8,0)}]
    \draw  \foreach \x in {20,140,...,380} {
        (\x:1) node{} -- (\x+120:1)};
     \draw {(140:1) node{} -- (140:0) node{}};
\end{scope}

      \tikzstyle{every node}=[inner sep=0,outer sep=0]
      
        \node at(-5,-0.2){$\Rightarrow$};
        \node at (-2,-0.2) {$\begin{psmallmatrix}0&1&1&1\\1&0&1&0\\1&1&0&0\\1&0&0&0\end{psmallmatrix}$};
        \node at(1,-0.2){$\Rightarrow$};
                
\node at (5,1) {$\begin{psmallmatrix} 0.3 & 0.2 & 1& -0.8\end{psmallmatrix}$};
\node at (5,0) {$\downarrow$};
\node at (5,-1) {$\begin{psmallmatrix} 0.4 & 1.3 & 0.5 & 0.3\end{psmallmatrix}$};

 \node at(9,-0.2){$\Rightarrow$};
 \node at(15,1){$\frac{1}{4}\Bigl(\delta_{(0.3,0.4)}+\delta_{(0.2,1.3)}$};
 \node at(17,-1){$+\delta_{(1,0.5)}+\delta_{(-0.8,0.3)}\Bigr)$};

\end{tikzpicture}\quad
\caption{graph $\Rightarrow$ operator $\Rightarrow$ action $\Rightarrow$ measure (computing an element in the $1$-profile of a graph)}
\end{figure}
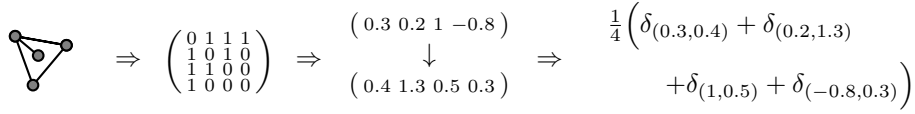

\noindent{\bf Adjacency operator convergence:}~We obtain a general graph convergence notion by considering the convergence of appropriately normalized adjacency matrices of graphs. For a graph $G$, let $A(G)$ denote the adjacency matrix of $G$. It turns out that for a bounded degree sequence of graphs $\{G_i\}_{i=1}^\infty$ the $P$-operator convergence of the sequence $\{A(G_i)\}_{i=1}^\infty$ is equivalent to local-global convergence (an thus it implies Benjamini--Schramm convergence). On the other hand, for a general graph sequence, the $P$-operator convergence of $A(G_i)/|V(G_i)|$ is equivalent to dense graph convergence. For graph sequences of intermediate growth we  normalize each operator $A(G_i)$ by a constant depending on $G_i$ to obtain non-trivial convergence notion and limit object. A natural choice is the  spectral radius given by $\|A(G_i)\|_{2\to 2}$ or, more generally, norms of the form $\|A(G_i)\|_{p\to q}$.

The convergence of normalized adjacency matrices leads to a rich limit theory for graphs of intermediate density. To demonstrate this we give various examples for convergent sequences and limit objects. We calculate the limit object of hypercube graphs. The hypercube graph $Q_n$ is the graph on $\{0,1\}^n$ in which two vectors are connected if they differ at exactly one coordinate. These graphs are very sparse and they are of intermediate density. The graph $Q_n$ is vertex-transitive and can be represented as a Cayley graph of the elementary abelian group $Z_2^n$ with respect to a minimal generating system. Quite surprisingly, the limiting $P$-operator turns out to be also a Cayley graph of the compact group $Z_2^\infty$ with respect to a carefully chosen topological generating system. This illustrates that our limit objects give natural representations of convergent sequences. We calculate similarly natural representations for other convergent graph sequences such as increasing powers of regular graphs and incidence graphs of projective planes.

\noindent{\bf Random walk metric and convergence:}~A possible limitation for the use of adjacency operator convergence is that it may trivialize if the degree distribution is very uneven in a graph sequence. The simplest examples are stars and subdivisions of complete graphs. In the star graph $S_n$ there is one vertex with degree $n$ and $n$ vertices with degree $1$. When normalized in any reasonable way, they converge to the $0$ operator. The property that a graph has very uneven degree distribution is related to the property that a random walk on the graph spends a positive proportion of the time in a negligible fraction of the vertex set. A natural way to counterbalance this problem is to use Markov kernels of random walks  instead of adjacency operators. (Such a modified limit was first used by Benjamini and Curien in case of bounded degree graphs \cite{curien}.) The $P$-operator language shows a nice advantage in this case to the plain matrix language. Even for finite graphs $G$ the corresponding Markov kernel is not just a matrix. The underlying probability space on $V(G)$ is modified from the uniform distribution to the stationary distribution $\nu_G$ of the random walk. Note that $\nu_G(i)$ is proportional to the degree $d(i)$ of $i\in V(G)$. The operator $M(G)$ is given by 
\begin{equation}\label{eq:vmg}(vM(G))(i)=d(i)^{-1}\sum_{(i,j)\in E(G)} v(j)\end{equation} for $i\in{\rm supp}(\nu_G)$. Although $M(G)$ is not symmetric when viewed as a matrix, its action on $L^2(V(G),\nu_G)$ is self-adjoint. Thus $M(G)$ is a positivity-preserving, self-adjoint $P$-operator with the property that $1_{V(G)}M(G)=1_{V(G)}$. The last property is called $1$-regularity.

\smallskip

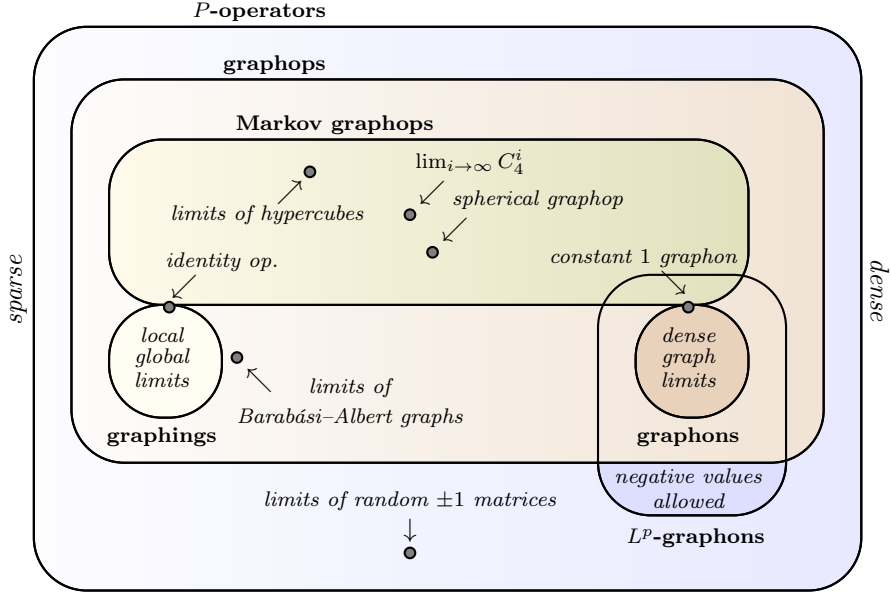
\begin{figure}
 \begin{tikzpicture}[thick]
        \draw [help lines, black!10] (-.5,-.5); %grid (7.5,3.5);
        
        \colorlet{lightbluel}{blue!1}
        \colorlet{lightbluer}{blue!9}
        \colorlet{lightbluee}{blue!13}
        \colorlet{MyColorOnel}{brown!2}
        \colorlet{MyColorOner}{brown!22}
        \colorlet{MyColorTwo}{brown!8}
        \colorlet{MyColorThree}{brown!18}
        \colorlet{MyColorFourl}{yellow!6}
        \colorlet{MyColorFourr}{brown!35}
        \definecolor{tempcolorl}{RGB}{255,250,235}
        \definecolor{tempcolorr}{RGB}{225,225,190}
        
        \filldraw[left color = lightbluel , right color= lightbluer , rounded corners=20] (0,-0.5) rectangle ++(11,7.5) ;
           \draw [fill=lightbluee,rounded corners=20] (7.5,0.5) rectangle ++(2.5,3.2); 
        \filldraw[left color = MyColorOnel, right color = MyColorOner ,  rounded corners=20] (0.5,1.2) rectangle ++(10,5.1); 
        \filldraw [left color=tempcolorl , right color = tempcolorr , rounded corners=20] (1,3.3) rectangle ++(8.5,2.2); 
     
        \draw [fill=MyColorFourl, rounded corners=20] (1,1.8) rectangle ++(1.5,1.5); 
        \draw [fill=MyColorFourr,rounded corners=20] (8,1.8) rectangle ++(1.5,1.5);

         \draw [rounded corners=20] (0,-0.5) rectangle ++(11,7.5) ;
        \draw [rounded corners=20] (0.5,1.2) rectangle ++(10,5.1); 
        \draw [rounded corners=20] (1,3.3) rectangle ++(8.5,2.2); 
        \draw [rounded corners=20] (1,1.8) rectangle ++(1.5,1.5); 
        \draw [rounded corners=20] (8,1.8) rectangle ++(1.5,1.5); 
        \draw [rounded corners=20] (7.5,0.5) rectangle ++(2.5,3.2); 
        
        %\draw [rounded corners=1cm] (4,0) rectangle ++(3,3) node [midway] };
        
        \node at(3,7.2) {\footnotesize\text{\bf $P$-operators}};
        \node at(3.2,6.5) {\footnotesize\text{\bf graphops}};
        \node at(4,5.7) {\footnotesize\text{\bf Markov graphops}};
        \node at(1.7,1.55) {\footnotesize\text{\bf graphings}};
        \node at(8.7,1.55) {\footnotesize\text{\bf graphons}};
        \node at(8.8,0.2){\footnotesize\text{\bf $L^p$-graphons}};
        
        \node at(2,3.55){$\swarrow$};
        \node at(3.4,4.8){$\nearrow$};
        \node at (8.5,3.6){$\searrow$};
        \node at (5,0.3){$\downarrow$};
        \node at (3,2.3){$\nwarrow$};
        \node at (5.3,4.8){$\swarrow$};
        \node at (5.6,4.3){$\swarrow$};
        \node at(2.5,3.88){\footnotesize\text{\it identity op.}};
        \node at(3.1,4.5){\footnotesize\text{\it limits of hypercubes}};
        \node at(8.1,3.93){\footnotesize\text{\it constant $1$ graphon}};
        \node at(5,0.7){\footnotesize\text{\it limits of random $\pm 1$ matrices}};
        \node at (4.2,2.2){\footnotesize\text{\it limits of}};
        \node at (4.2,1.8){\footnotesize\text{\it Barab\'asi--Albert graphs}};
        \node at (5.8,5.2){\footnotesize\text{\it $\lim_{i\to\infty} C_4^i$}};
        \node at (8.7,1){\footnotesize\text{\it negative values}};
        \node at (8.7,0.7){\footnotesize\text{\it allowed}};
        \node [rotate=90] at (-0.2,3.5){\text{\it sparse}};
        \node [rotate=270] at (11.25,3.5){\text{\it dense}};
        \node at (6.7,4.7) {\footnotesize\text{\it spherical graphop}};
        \node at (8.7,2.9) {\footnotesize\text{\it dense}};
        \node at (8.7,2.6) {\footnotesize\text{\it graph}};
        \node at (8.7,2.3) {\footnotesize\text{\it limits}};
        \node at (1.7,2.9) {\footnotesize\text{\it local}};
        \node at (1.7,2.6) {\footnotesize\text{\it global}};
        \node at (1.7,2.3) {\footnotesize\text{\it limits}};
        
        \tikzstyle{every node}=[circle, draw, fill=black!50,
                        inner sep=0pt, minimum width=4pt] ;    
                        
        \node at(1.8,3.27) {};   
        \node at(8.7,3.27) {};
        \node at(3.67,5.07) {};
        \node at(5,0) {};
        \node at(2.7,2.6) {};
        \node at(5,4.5) {};
        \node at(5.3,4) {};
        
    \end{tikzpicture}
    \caption{Universe of $P$-operators}
    \end{figure}

\noindent{\it The random walk metric $d_{\rm RW}$ on finite non-empty graphs is given by $$d_{\rm RW}(G_1,G_2):=d_M(M(G_1),M(G_2)).$$ The completion of the set of finite, non-empty graphs in $d_{\rm RW}$ is a compact space $\mathcal{G}_{\rm RW}$. Elements of $\mathcal{G}_{\rm RW}$ can be represented by {\bf Markov graphops}, i.e. positivity-preserving, self-adjoint $1$-regular $P$-operators.}

\smallskip

Note that if $A$ is a Markov graphop, then $\|A\|_{2\to 2}=1$. We will see that Markov graphops can also be represented by symmetric self-couplings of probability spaces $(\Omega,\mathcal{A},\mu)$. A symmetric self-coupling is a probability measure $\nu$ on $(\Omega\times\Omega,\mathcal{A}\otimes\mathcal{A})$ such that $\nu$ is symmetric with respect to interchanging the coordinates and both marginals of $\nu$ are equal to $\mu$.
A very pleasant property of the set of all Markov graphops is that it is compact in the metric $d_M$ (see Theorem \ref{markcomp}), and thus we do not need any extra conditions to guarantee convergent subsequences.

We will show in the examples section (Section \ref{chapex}) that stars and subdivisions of complete graphs converge to natural and non-trivial limit objects according to random walk convergence. Note that random walk convergence coincides with normalized adjacency operator convergence for regular graphs (graphs in which every degree is the same).

As we mentioned before, random walk convergence is very convenient. Every graph sequence $\{G_i\}_{i=1}^\infty$ with non-empty graphs has a convergent subsequence in $d_{\rm RW}$, and  the limit object is usually an interesting structured object independently of the sparsity of the sequence. The most trivial the limit object one can get is the quasi-random Markov graphop which can be represented by the constant $1$ graphon $W(x,y):=1$. This occurs for example if the second largest eigenvalue in absolute value of $M(G_i)$ is $o(1)$.

\noindent{\bf Extended random walk convergence:} Finally we describe a general convergence notion that combines the advantages of adjacency operator convergence and random walk convergence. A feature of random walk convergence is that some information may be lost in the limit regarding degree distributions. It turns out that there is a rather natural way to solve this problem, using a mild extension of random walk convergence based on a simultaneous version of action convergence. The principle of action convergence allows us to introduce the convergence of pairs $(A,f)$ where  $A:L^\infty(\Omega)\to L^1(\Omega)$ is a $P$-operator and $f$ is a measurable function on $\Omega$. Roughly speaking, this goes by considering $f$ as a reference function that is automatically included as the last function into every function system used in the definition of the $k$-profile of $A$. More precisely, we define $S_k(A,f)$ as the set of all possible joint distributions of the random variables $\{v_i\}_{i=1}^k$, $\{v_iA\}_{i=1}^k$  and $f$ on $\Omega$, where the values of $v_i$ are in $[-1,1]$. We have that $S_k(A,f)$ is a set of probability measures on $\mathbb{R}^{2k+1}$. 

We can use the extra function to store information on the degrees of vertices in $G$. For a graph $G$ let $d^*_G$ denote a function on $V(G)$ that is an appropriately normalized version of the degree function $d_G$. We can represent $G$ by the pair $(M(G),d^*_G)$ (recall equation (\ref{eq:vmg})). In the limit we obtain a similar pair $(A,f)$ where $A$ is a Markov graphop. The non-negative function $f^{-1}$ (which may also take the value $\infty$) can be used to "re-scale" the probability measure on $\Omega$ to a possibly infinite measure. 

Pairs of the form $(A,f)$ can also be used to represent generalized graphons of the form $W:\mathbb{R}^+\times\mathbb{R}^+\to [0,1]$, where $W$ is symmetric and $\|W\|_1\leq\infty$. This construction will be described in Section \ref{chapgenlim}. Note that these generalized graphons arise in the recently emerging theory of graphexes \cite{BCCH, janson}. 

\begin{figure}
\begin{tikzpicture}
\node at (5.7,2){\text{$\Omega=\{(x,y,z):x^2+y^2+z^2=1\}~,~\mu=$ {\it uniform measure}}};

\node at  (3.7,1.2){\text{$(a,b)\in E \Leftrightarrow  a \bot b$}}; 
\node at (7.5,1.2){\text{$(\Omega,E)$ {\it is a Borel graph}}};
\node at (5.9,0.4){\text{$\nu_a:=$ {\it uniform measure on} $\{b: b\in\Omega , a\bot b\}$}};
\node at  (4.5,-0.4){\text{$S:L^2(\Omega,\mu)\to L^2(\Omega,\mu)~~$}};
\node at  (4.5,-1.2){\text{$(fS)(a):=\displaystyle\int_{b \bot a} f(b)~d\nu_a$}};
\node at (0,2.3){\text{$a$}};
\node at (-1.1,-0.18){\text{$b$}};

  \shade[ball color = gray!40, opacity = 0.4] (0,0) circle (2cm);
  \draw (0,0) circle (2cm);
  \draw (-2,0) arc (180:360:2 and 0.6);
  \draw[dashed] (2,0) arc (0:180:2 and 0.6);
  \fill[fill=black] (0,0) circle (1pt);
  \draw[dashed] (0,0 ) -- (0,2);
  \tikzstyle{every node}=[circle, draw, fill=black!50,
                        inner sep=0pt, minimum width=4pt] ; 
  \draw{(0,2) node{} -- (1.8,-0.28) node{}};
   \draw{(0,2) node{} -- (-1,-0.5) node{}};
   \draw[dashed] {(0,2) node{} -- (-1.5,0.38) node {}};
   
\end{tikzpicture}

\caption{{\it The spherical graphop.} It is neither a graphon nor a graphing: it is somewhere half way in between.}\label{sphericgraph}

\end{figure}
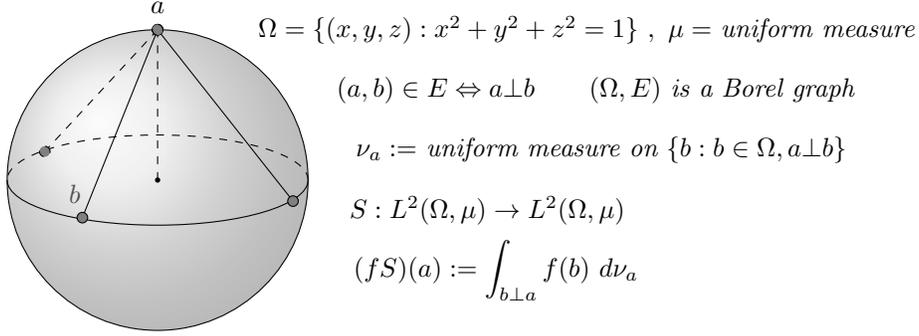

\noindent{\bf Applications to random matrix theory:}~As we mentioned earlier, the notion of $P$-operator convergence was partially motivated by efforts to find a fine enough convergence notion such that random matrices converge to a structured, non-trivial object. The study of this limiting object can help in describing approximate properties of random matrices such as entry distributions of eigenvectors and almost eigenvectors.  In dense graph limit theory random matrices with i.i.d $\pm 1$ entries converge, but the limit object is the constant $0$ function. (In a refinement  of dense limit theory \cite{kunszenti} the limit object is the constant function on $[0,1]^2$ whose value is the uniform probability measure on $\{1,-1\}$.)

Our main observation about random matrices is that $k$-profiles of appropriately normalized random matrices are non-trivial rich objects and their study brings a new point of view on random matrix theory. Let $\mathbb{G}_n$ denote a random matrix whose entries are i.i.d. zero-mean $\pm 1/\sqrt{n}$-valued random variables. The normalizing factor $\sqrt{n}$ is needed to obtain bounded spectral radius. With probability close to $1$, we have that $\|\mathbb{G}_n\|_{2\to 2}$ is close to $2$, see e.g.\ \cite{geman}. In this paper we do ground work on the limiting properties of $\mathbb{G}_n$ according to action convergence. We prove a {\it concentration of measure} type statement for $\mathbb{G}_n$ with respect to the metric $d_M$. This means that for large $n$ the matrix-valued random variable $\mathbb{G}_n$ is well concentrated in the metric space of $P$-operators together with distance $d_M$. This concentration result together with our compactness results imply that for certain good sequences $\{n_i\}_{i=1}^\infty$ of natural numbers, $\{\mathbb{G}_{n_i}\}_{i=1}^\infty$ is convergent with probability one and the limit object is represented by some $P$-operator $L$ acting on $L^2([0,1])$ with bounded $\|.\|_{2\to 2}$ norm. In this paper we leave the question open whether the sequence of all natural numbers is a good sequence. Note that a similar open problem is know for random regular graphs using local-global convergence \cite{large} and it is known that a positive answer would imply the convergence of a great number of interesting graph parameters. For our application it will be enough that any sequence of natural numbers contains a good sub-sequence. Our general results in this paper prepare a follow-up paper which focuses on the limiting properties of random matrices with a special emphasis on eigenvectors and almost eigenvectors.

\section{Limits of matrices and operators} \label{chaplim}

For $k$ vectors $\{v_i\in\mathbb{R}^n\}_{i=1}^k$ let us define their joint empirical entry distribution,  denoted by $\mathcal{D}(v_1,v_2,\dots,v_k)$, as the probability measure on $\mathbb{R}^k$ given by
\begin{equation}\label{empeq}
\mathcal{D}(v_1,v_2\dots,v_k):=n^{-1}\sum_{j=1}^n \delta_{(v_{1,j},v_{2,j},\dots,v_{k,j})},
\end{equation}
where $v_{i,j}$ denotes the $j$-th component of $v_i$ and $\delta_x$ denotes the Dirac measure at $x\in\mathbb{R}^k$. A natural view on empirical entry distributions is the following. Consider $[n]:=\{1,2,\dots,n\}$ as a probability space with the uniform distribution $\mu_{[n]}$ and vectors in $\mathbb{R}^n$ as functions of the form $v:[n]\to\mathbb{R}$. From this view point vectors are random variables and matrices in $\mathbb{R}^{n\times n}$ are operators acting on the space of random variables on the probability space $([n],\mu_{[n]})$. The joint empirical entry distribution $\mathcal{D}(v_1,v_2,\dots,v_k)$ is simply the joint distribution of the vectors $v_1,v_2,\dots,v_n$ viewed as random variables. 

Let $(\Omega,\mathcal{A},\mu)$ (or shortly $\Omega$) be a probability space and assume that $v_1,v_2,\dots,v_k$ are $\mathbb{R}$-valued measurable functions on $\Omega$. We denote by $\mathcal{D}(v_1,v_2,\dots,v_k)$ the joint distribution of $v_1,v_2,\dots,v_k$. In other words it is the push-forward of the measure $\mu$ under the map $x\mapsto (v_1(x),v_2(x),\dots,v_k(x))$, which is a Borel measure on $\mathbb{R}^k$. 

\begin{definition}\label{def:pop} A {\bf $P$-operator} is a linear operator of the form $A:L^\infty(\Omega)\to L^1(\Omega)$ such that $$\|A\|_{\infty\to 1}:=\sup_{v\in L^\infty(\Omega)}\|vA\|_1/\|v\|_\infty$$ is finite. We denote by $\mathcal{B}(\Omega)$ the set of all $P$-operators on $\Omega$. 
\end{definition}

\begin{remarka} If $\Omega=[n]$ and $\mu=\mu_{[n]}$, then $L^1(\Omega)=L^\infty(\Omega)=\mathbb{R}^n$. In this case $\mathcal{B}(\Omega)$ is the set of all $n\times n$ matrices. Thus every matrix $A\in\mathbb{R}^{n\times n}$ is a $P$-operator. 
\end{remarka}

For a set $S\subseteq\mathbb{R}$ we denote by $L^\infty_S(\Omega)$ the set of bounded measurable functions on $\Omega$ whose values are in $S$.

\begin{definition} {\bf ($k$-profile of $P$-operators)} For a $P$-operator $A\in\mathcal{B}(\Omega)$ we define the $k$-profile of $A$, denoted by $\mathcal{S}_k(A)$, as the set of all possible probability measures of the form
\begin{equation}\label{jointeq}
\mathcal{D}(v_1,v_2,\dots,v_k,v_1A,v_2A,\dots,v_kA),
\end{equation}
where $v_1,v_2,\dots,v_k$ run through all possible $k$-tuples of functions in $L^\infty_{[-1,1]}(\Omega)$.
\end{definition}    

For joint distributions of the form (\ref{jointeq}) we will often use the shothand notation
$$\mathcal{D}_A(v_1,v_2,\dots,v_k):=\mathcal{D}(v_1,v_2,\dots,v_k,v_1A,v_2A,\dots,v_kA).$$
Let $\mathcal{P}(\mathbb{R}^k)$ denote the set of Borel probability measures on $\mathbb{R}^k$ for $k\in\mathbb{N}$.

\begin{definition} {\bf (L\'evy--Prokhorov metric)} The L\'evy--Prokhorov metric $d_{\rm LP}$ on $\mathcal{P}(\mathbb{R}^k)$ is defined by
 \[d_{\rm LP}(\eta_1, \eta_2)=\inf\{\varepsilon>0: \eta_1(U)\leq \eta_2(U^{\varepsilon})+\varepsilon \text{ and } \eta_2(U)\leq \eta_1(U^{\varepsilon})+\varepsilon \text{ for all } U\in\mathcal{B}_k\},\]
where $\mathcal{B}_k$ is the Borel $\sigma$-algebra on $\mathbb{R}^k$ and $U^{\varepsilon}$ is the set of points that have distance smaller than $\varepsilon$ from $U$.
\end{definition}

\begin{definition}{\bf (Hausdorff metric)} We measure the distance of subsets $X,Y$ in $\mathcal{P}(\mathbb{R}^k)$ using the Hausdorff metric $d_H$. $$d_H(X,Y):=\max\Big\{\sup_{x\in X}\inf_{y\in Y} d_{\rm LP}(x,y)~,~\sup_{y\in Y}\inf_{x\in X} d_{\rm LP}(x,y)\Big\}.$$ Note that $d_H(X,Y)=0$ if and only if ${\rm cl}(X)={\rm cl}(Y)$, where ${\rm cl}$ is the closure in $d_{\rm LP}$. It follows that $d_H$ is a pseudometric for all subsets in $\mathcal{P}(\mathbb{R}^k)$ and it is a metric for closed sets.
\end{definition}

\begin{definition}\label{limconv} {\bf (Action convergence of $P$-operators)} We say that a sequence of $P$-operators $\{A_i\in\mathcal{B}(\Omega_i)\}_{i=1}^\infty$ is action convergent if for every $k\in\mathbb{N}$ the sequence $\{\mathcal{S}_k(A_i)\}_{i=1}^\infty$ is a Cauchy sequence in $d_H$. 
\end{definition}

\begin{remarka} The completeness of $(\mathcal{P}(\mathbb{R}^k),d_H)$ implies that $\{A_i\}_{i=1}^\infty$ in the above definition is convergent if and only if for every $k\in\mathbb{N}$ there is a closed set $X_k$ such that $\lim_{i\to\infty} d_H(\mathcal{S}_k(A_i),X_k)=0$.
\end{remarka}

\begin{definition} {\bf (Metrization of action convergence)} For two $P$-operators $A,B$ let 
\begin{equation}\label{dmetric}
d_M(A,B):=\sum_{k=1}^\infty 2^{-k} d_H(\mathcal{S}_k(A),\mathcal{S}_k(B)).
\end{equation}
\end{definition}

For a probability measure $\mu$ on $\mathbb{R}^k$ let $\tau(\mu)\in [0,\infty]$ denote the quantity \begin{equation}\label{eq:tau}\max_{1\leq i\leq k} \int_{(x_1,x_2,\dots,x_k)\in\mathbb{R}^k}|x_i|~d\mu.\end{equation} For $c\in\mathbb{R}^+$ and $k\in\mathbb{N}$ let 
$$\mathcal{P}_c(\mathbb{R}^k):=\{\mu : \mu\in\mathcal{P}(\mathbb{R}^k) , \tau(\mu)\leq c\}.$$
Let furthermore $\mathcal{Q}_c(\mathbb{R}^k)$ denote the set of closed sets in the metric space $(\mathcal{P}_c(\mathbb{R}^k),d_{\rm LP})$. The following lemma is an easy consequence of classical results. 

\begin{lemma}\label{limitlem1} The metric spaces $(\mathcal{P}_c(\mathbb{R}^k),d_{\rm LP})$ and $(\mathcal{Q}_c(\mathbb{R}^k),d_H)$ are both compact and complete metric spaces.
\end{lemma}

\begin{proof} Markov's inequality gives uniform tightness in $\mathcal{P}_c(\mathbb{R}^k)$, which implies the compactness of $(\mathcal{P}_c(\mathbb{R}^k),d_{\rm LP})$. It is known that Hausdorff distance for the closed subsets in a compact Polish space is again compact.
\end{proof}

\begin{lemma}\label{limitlem2} Let $A\in\mathcal{B}(\Omega)$ and let $c:=\max(\|A\|_{\infty\to1},1)$. Then for every $k\in\mathbb{N}$ we have that $\mathcal{S}_k(A)\in \mathcal{Q}_c(\mathbb{R}^{2k})$.
\end{lemma}

\begin{proof} Let $\{v_i\}_{i=1}^k$ be a system of vectors in $L^\infty_{[-1,1]}(\Omega)$. We have that $\|v_i\|_1\leq \|v_i\|_\infty\leq 1$ and $\|v_iA\|_1\leq \|A\|_{\infty\to 1}$  holds for $1\leq i\leq k$. Since the first moments of the absolute values of the coordinates in (\ref{jointeq}) are given by $\{\|v_i\|_1\}_{i=1}^k$ and $\{\|v_iA\|_1\}_{i=1}^k$,  the proof is complete.
\end{proof}

Lemma \ref{limitlem1} and Lemma \ref{limitlem2} have the following corollary.

\begin{lemma} \label{lem:sc} {\bf (Sequential compactness)} Let $\{A_i\}_{i=1}^\infty$ be a sequence of $P$-ope\-rators with uniformly bounded $\|.\|_{\infty\to 1}$ norms. Then it has an action convergent subsequence.
\end{lemma}

For a real number $p\in [1,\infty]$ and measurable function $v:\Omega\to\mathbb{R}$ we have that $$\|v\|_p:=\Bigl(\int |v|^p~d\mu\Bigr)^{1/p}\in \mathbb{R}^+\cup\{0,\infty\}.$$ Note that if $p=\infty$, then $\|v\|_\infty$ denotes the "essential maximum" of $v$.
It is well known that for $p\leq q$ we have $\|v\|_p\leq\|v\|_q$ for any measurable function $v$ on $\Omega$.  
Let $p,q\in [1,\infty]$ be real numbers and let $A:L^\infty(\Omega)\to L^1(\Omega)$ be a linear operator. The operator norm $\|A\|_{p\to q}$ is defined by $$\|A\|_{p\to q}:=\sup_{v\in L^\infty(\Omega)}\|vA\|_q/\|v\|_p.$$ We say that $A$ is $(p,q)$-bounded if $\|A\|_{p\to q}$ is finite. We have that if $p',q'\in [1,\infty]$ satisfy $p'\geq p$ and $q'\leq q$ then $\|A\|_{p'\to q'}\leq \|A\|_{p\to q}$. We denote by $\mathcal{B}_{p,q}(\Omega)$ the set of $(p,q)$-bounded linear operators from $L^\infty(\Omega)$ to $L^1(\Omega)$.  If $p',q'\in [1,\infty]$ satisfy $p'\geq p$ and $q'\leq q$, then $\mathcal{B}_{p,q}(\Omega)\subseteq\mathcal{B}_{p',q'}(\Omega)$. In particular, $\mathcal{B}_{p,q}(\Omega)\subseteq\mathcal{B}_{p,1}(\Omega)$ and $\mathcal{B}(\Omega)=\mathcal{B}_{\infty,1}(\Omega)$ contains $\mathcal{B}_{p,q}$ for every $p,q\in [1,\infty]$. %If $p=q$ then we use the short hand notation $\|A\|_p$ for $\|A\|_{p\to p}$ and $\mathcal{B}_p(\Omega)$ for $\mathcal{B}_{p,p}(\Omega)$. 

\begin{remarka} {\bf ($L^2$ theory)} If a $P$-operator $A$ satisfies $\|A\|_{p\to q}\leq\infty$, then $A$ extends uniquely to an operator of the form $L^p(\Omega)\to L^q(\Omega)$. In this sense, by slightly abusing the notation, we can identify the set $\mathcal{B}_{p,q}(\Omega)$ with the set of operators $L^p(\Omega)\to L^q(\Omega)$ with bounded $\|.\|_{p\to q}$ norm. An especially nice class of $P$-operators is the set $\mathcal{B}_{2,2}(\Omega)$. If $\Omega$ is fixed then these operators are closed with respect to composition and they form a so-called von Neumann algebra. 
\end{remarka}

\begin{remarka} {\bf ($L^p$ graphons as $P$-operators)} Let $p\in (0,\infty]$ and assume that $W:\Omega\times\Omega\to\mathbb{R}$ is a measurable function such that $$\|W\|_p:=\Bigl(\int |W|^p~d\mu^2\Bigr)^{1/p}<\infty.$$ Let $q:=p/(p-1)$. We can associate an operator $A_W:L^q(\Omega)\to L^p(\Omega)$ with $W$,  defined by $(fA_{W})(x):=\int f(y)W(y,x)~d\mu$. It is easy to see that $\|A_W\|_{q\to p}<\infty$ and thus $A_W\in\mathcal{B}_{q,p}$ is a $P$-operator representing the so-called $L^p$-graphon $W$. For a theory of $L^p$-graphon convergence see \cite{BCCZ}. It follows from our theory that for sequences of $L^p$-graphons $\{W_i\}_{i=1}^\infty$ with uniformly bounded $L^p$-norms, action convergence of the representing operators $A_{W_i}$ is equivalent to $L^p$-graphon convergence. 
\end{remarka}

The following theorem is one of the main results in this paper. Its proof can be found in Section \ref{proofexlim}. 

\begin{theorem}\label{exlim} {\bf (Existence of limit object)} Let $p\in [1,\infty)$ and $q\in [1,\infty]$. Let $\{A_i\}_{i=1}^\infty$ be a convergent sequence of $P$-operators with uniformly bounded $\|.\|_{p\to q}$ norms. Then there is a $P$-operator $A$ such that $\lim_{i\to\infty} d_M(A_i,A)=0$, and $\|A\|_{p\to q}\leq\sup_{i\in\mathbb{N}}\|A_i\|_{p\to q}$.
\end{theorem}

For a $P$-operator $A$ and $k\in\mathbb{N}$ let $\mathcal{S}_k^*(A)$ denote the closure of $\mathcal{S}_k(A)$ in the space $(\mathcal{P}(\mathbb{R}^{2k}),d_{\rm LP})$.

\begin{definition} {\bf(Weak equivalence and weak containment)} Let $A$ and $B$ be two $P$-operators. We say that $A$ and $B$ are  weakly equivalent if $d_M(A,B)=0$. We have that $A$ and $B$ are weakly equivalent if and only if $\mathcal{S}_k^*(A)=\mathcal{S}_k^*(B)$ holds for every $k\in\mathbb{N}$. We say that $A$ is weakly contained in $B$ if $\mathcal{S}_k^*(A)\subseteq\mathcal{S}_k^*(B)$ holds for every $k\in\mathbb{N}$. We denote weak containment by $A\prec B$.
\end{definition}

It is easy to see that norms of the form $\|.\|_{p\to q}$ are invariant with respect to weak equivalence (these norms can be read off from the $1$-profiles of $P$-operators.) Let $\mathcal{X}$ denote the set of weak equivalence classes of $P$-operators and let $\mathcal{X}'\subset\mathcal{X}$ denote the set of equivalence classes of $P$-operators defined on atomless probability spaces.  For $c\in\mathbb{R}^+, p\in [1,\infty),q\in (1,\infty]$ let  $\mathcal{X}_{p,q,c}:=\{A: A\in\mathcal{X}, \|A\|_{p\to q}\leq c\}$ and $\mathcal{X}'_{p,q,c}:=\{A: A\in\mathcal{X}', \|A\|_{p\to q}\leq c\}$. 

The next theorem follows from Theorem \ref{exlim} and Lemma \ref{atclosed}. 

\begin{theorem}  \label{thm:compact} {\bf (Compactness)} For every $c\in\mathbb{R}^+,p\in [1,\infty),q\in (1,\infty]$ the spaces $(\mathcal{X}_{p,q,c},d_M)$ and $(\mathcal{X}'_{p,q,c},d_M)$ are compact. 
\end{theorem}

\begin{corollary} Let $c\in\mathbb{R}^+$. We have that $\{W:W\in \mathcal{B}_{2\to 2}([0,1]) , \|W\|_2\leq c\}$ is compact in the topology generated by $d_M$.
\end{corollary}

\begin{lemma}\label{limlem3} Let $k\in\mathbb{N}$ and let $A,B$ be $P$-operators both in $\mathcal{B}(\Omega)$ for some probability space $(\Omega,\mathcal{A},\mu)$. Then $$d_H(\mathcal{S}_k(A),\mathcal{S}_k(B))\leq \|A-B\|_{\infty\to 1}^{1/2}(2k)^{3/4}.$$
\end{lemma}

\begin{proof} Let $x\in \mathcal{S}_k(A)$ be arbitrary. We have that there are functions $v_1,v_2,\dots,v_k\in L^\infty_{[-1,1]}(\Omega)$  such that $x$ is equal to the probability measure $\mathcal{D}_A(\{v_i\}_{i=1}^k)$. Let $y:=\mathcal{D}_B(\{v_i\}_{i=1}^k)\in\mathcal{S}_k(B)$. Since  $$\|v_iA-v_iB\|_1\leq \|v_i\|_\infty\|A-B\|_{\infty\to 1}\leq\|A-B\|_{\infty\to 1}$$ holds for every $i\in [k]$, we have by Lemma \ref{coupdist2} that $d_{\rm LP}(x,y)\leq \|A-B\|_{\infty\to 1}^{1/2}(2k)^{3/4}$. We obtained that $$\sup_{x\in\mathcal{S}_k(A)}\inf_{z\in\mathcal{S}_k(B)}d_{\rm LP}(x,z)\leq \|A-B\|_{\infty\to 1}^{1/2}(2k)^{3/4}.$$ By switching the roles of $A$ and $B$ and repeating the same argument we get the above inequality with $A$ and $B$ switched. This implies the statement of the lemma.
\end{proof}

\begin{lemma}\label{limspdm} {\bf (Norm distance vs. $d_M$ distance)} Assume that $A,B$ are $P$-operators acting on the same space $L^\infty(\Omega)$. We have  $d_M(A,B)\leq 3\|A-B\|_{\infty\to 1}^{1/2}$.
\end{lemma}

\begin{proof} Using Lemma \ref{limlem3} we obtain that 
\begin{equation*}d_M(A,B)\leq \|A-B\|_{\infty\to 1}^{1/2}\sum_{k=1}^\infty 2^{-k}(2k)^{3/4}\leq 3\|A-B\|_{\infty\to 1}^{1/2}.\end{equation*}
\end{proof}

Let $A\in\mathcal{B}(\Omega,\mathcal{A},\mu)$ be a $P$-operator. We define the bilinear form  $(f,g)_A$ for functions $f,g\in L^\infty(\Omega)$ by
\begin{equation}\label{skaleq}
(f,g)_A:=\int_{\Omega}~ (fA)g~d\mu = \mathbb{E}((fA)g).
\end{equation}
 Note that $$(f,g)_A\leq \|f\|_\infty\|g\|_\infty\|A\|_{\infty\to 1},$$ and thus $(f,g)_A$ is finite. In general if $\|A\|_{p\to q}\leq\infty$ holds for a conjugate pair with $1/p+1/q=1$, then we have 
\begin{equation}\label{skalineq}
 (f,g)_A\leq\|f\|_p\|g\|_p\|A\|_{p\to q}.
\end{equation}

 We define the cut norm of $A$ by 
$$\|A\|_\square:=\sup_{S,T\in\mathcal{A}} ~ |(1_S,1_T)_A|.$$
It is well known that 
$$\|A\|_\square\leq \|A\|_{\infty\to 1}\leq 4\|A\|_\square,$$ which means that $\|.\|_\square$ is equivalent to the norm $\|.\|_{\infty\to 1}$.
Let $\psi:\Omega\to\Omega$ be an invertible measure-preserving transformation with measure-preserving inverse. The transformation $\psi$ induces a natural, linear action on $L^\infty(\Omega)$, also denoted by $\psi$, defined by $f^\psi(x):=f(\psi(x))$. Furthermore, for $A\in\mathcal{B}(\Omega)$ let $A^\psi:=\psi^{-1}\circ A\circ\psi$. It is easy to see that if $A\in\mathcal{B}(\Omega)$, then $A^\psi\in\mathcal{B}(\Omega)$ and that $d_M(A,A^\psi)=0$. Let
$$\delta_\square(A,B):=\inf_{\psi,\phi}~\|A^\psi-B^\phi\|_\square,$$ where $\psi,\phi$ run through all invertible measure preserving transformations of $\Omega$. The proof of the next lemma follows from Lemma \ref{limspdm}.

\begin{lemma}\label{limspdm2} Assume that $A,B$ are $P$-operators acting on the same space $L^\infty(\Omega)$. Then $d_M(A,B)\leq 12\delta_\square(A,B)^{1/2}$.
\end{lemma}

\section{$P$-operators with special properties}

The goal of this chapter is to show that various fundamental properties behave well with respect to $P$-operator convergence. 

\begin{definition}\label{defprops} Let $A\in\mathcal{B}(\Omega)$ be a $P$-operator. 
\begin{itemize}
\item $A$ is {\bf self-adjoint} if $(v,w)_A=(v,w)_A$ holds for every $v,w\in L^\infty(\Omega)$.
\item $A$ is {\bf positive} if $(v,v)_A\geq 0$ holds for every $v\in L^\infty(\Omega)$.
\item $A$ is {\bf positivity-preserving} if for every $v\in L^\infty(\Omega)$ with $v(x)\geq 0$ for almost every $x\in\Omega$, we have that $(vA)(x)\geq 0$ holds for almost every $x\in\Omega$.
\item $A$ is {\bf $c$-regular} if $1_\Omega A=c1_\Omega$ for some $c\in\mathbb{R}$.
\item $A$ is a {\bf graphop} if it is positivity-preserving and self-adjoint.
\item $A$ is a {\bf Markov graphop} if $A$ is a $1$-regular graphop.
\item $A$ is {\bf atomless} if $\Omega$ is atomless.
\end{itemize}
\end{definition}

\begin{lemma}\label{atclosed} Atomless $P$-operators are closed with respect to $d_M$.
\end{lemma}

\begin{proof} Let $A\in\mathcal{B}(\Omega)$ be an atomless $P$-operator and let $B\in\mathcal{B}(\Omega_2)$ with $d_M(B,A)=\varepsilon$. We have that there is a function $v\in L^{\infty}_{[-1,1]}(\Omega)$ such that the distribution of $v$ is uniform on $[-1,1]$. Let $\alpha:=\mathcal{D}_A(v)$.  It follows from $d_H(\mathcal{S}_1(A),\mathcal{S}_1(B))\leq 2\varepsilon$ that there is $\beta=\mathcal{D}_B(w)\in\mathcal{S}_1(B)$ with $d_{\rm LP}(\beta,\alpha)\leq 3\varepsilon$ and thus $d_{\rm LP}(\alpha_1,\beta_1)\leq3\varepsilon$, where $\alpha_1=\mathcal{D}(v)$ and $\beta_1=\mathcal{D}(w)$ are the marginals of $\alpha$ and $\beta$ on the first coordinate. It follows that $\beta_1$ is at most $3\varepsilon$ far from the uniform distribution in $d_{\rm LP}$, and thus the largest atom in $\beta_1$ is at most $10\varepsilon$. Hence the largest atom in $\Omega_2$ has weight at most $10\varepsilon=10d_M(B,A)$. We obtained that if $B$ is the limit of atomless operators, then $B$ is atomless. 
\end{proof}

\begin{proposition}\label{closedprop} Let $p\in[1,\infty]$ and $q\in (1,\infty)$. Let $\{A_i\in\mathcal{B}(\Omega_i)\}_{i=1}^\infty$ be a sequence of uniformly $(p,q)$-bounded  $P$-operators converging to a $P$-operator $A\in\mathcal{B}(\Omega)$. Then we have the following two statements.
\begin{enumerate}
\item If $A_i$ is positive for every $i$, then $A$ is also positive.
\item If $A_i$ is self-adjoint for every $i$, then $A$ is also self-adjoint.
\end{enumerate}
\end{proposition}

\begin{proof} To prove the first claim let $v\in L^\infty_{[-1,1]}(\Omega)$. By the definition of convergence (Definition \ref{limconv}) there are elements $v_i\in L^\infty_{[-1,1]}(\Omega_i)$  such that $\mathcal{D}_A(v_i)$ weakly converges to $\mathcal{D}_A(v)$ as $i$ goes to infinity. By Lemma \ref{closedlem2} we have that $(v_i,v_i)_{A_i}=\mathbb{E}(v_i(A_iv_i))$ converges to $(v,v)_A=\mathbb{E}(v(Av))$. We have by the  assumption that $(v_i,v_i)_{A_i}\geq 0$ holds for every $i$, and thus $(v,v)_A\geq 0$ holds. 

To prove the second claim let $v,w\in L^\infty_{[-1,1]}(\Omega)$ and let $\mu:=\mathcal{D}_A(v,w)$. By the definition of convergence we have that for every $i\in\mathbb{N}$ there exist functions $v_i,w_i\in L_{[-1,1]}^\infty(\Omega_i)$ such that $\mu_i:=\mathcal{D}_{A_i}(v_i,w_i)$ weakly converges to $\mu$.  By Lemma \ref{closedlem2} we obtain that $\mathbb{E}(v_i(w_iA_i))$ weakly converges to $\mathbb{E}(v(wA))$ and $\mathbb{E}((v_iA_i)w_i)$ weakly converges to $\mathbb{E}((vA)w)$ as $i$ goes to infinity. On the other hand we have that $$\mathbb{E}(v_i(w_iA_i))=(w_i,v_i)_{A_i},\qquad \mathbb{E}(v(wA))=(w,v)_A,$$ $$\mathbb{E}((v_iA_i)w_i)=(v_i,w_i)_{A_i},\qquad \mathbb{E}((vA)w)=(v,w)_A.$$
Since by assumption we have $(v_i,w_i)_{A_i}=(w_i,v_i)_{A_i}$ the proof is complete. 
\end{proof} 

\begin{proposition}\label{closedprop2} Let $p\in [1,\infty),q\in [1,\infty],c\in\mathbb{R}$ and let $\{A_i\in\mathcal{B}(\Omega_i)\}_{i=1}^\infty$ be a sequence of uniformly $(p,q)$-bounded  $P$-operators converging to a $P$-operator $A\in\mathcal{B}(\Omega)$. Then we have the following two statements.
\begin{enumerate}
\item If $A_i$ is positivity-preserving for every $i$, then $A$ is also positivity-preserving.
\item If $A_i$ is $c$-regular for every $i$, then $A$ is also $c$-regular.
\end{enumerate}
\end{proposition}

\begin{proof} Let $v\in L^\infty_{[0,1]}(\Omega)$. Then there is a sequence $\{v_i\in L_{[-1,1]}^\infty(\Omega_i)\}_{i=1}^\infty$  such that $\mathcal{D}_{A_i}(v_i)$ weakly converges to $\mathcal{D}_A(v)$. The fact that $\mathcal{D}(v_i)$ weakly converges to the non-negative distribution $\mathcal{D}(v)$ implies that $\mathcal{D}(v_i-|v_i|)$ weakly converges to $\delta_0$. It follows from Lemma \ref{applem2} that $d_{\rm LP}(\mathcal{D}_{A_i}(v_i),\mathcal{D}_{A_i}(|v_i|))$ converges to $0$ and so $(v,vA)$ is the weak limit of $\mathcal{D}_{A_i}(|v_i|)$. Since $|v_i|A_i$ is non-negative for every $i$  we obtain that $vA$ is non-negative. 

Let $v_i\in L_{[-1,1]}^\infty(\Omega_i)$ be a sequence of functions such that $\mathcal{D}_{A_i}(v_i)$ weakly converges to $\mathcal{D}_A(1_\Omega)$. We have that $\mathcal{D}(v_i-1_{\Omega_i})$ weakly converges to $\delta_0$ and hence by Lemma \ref{applem2} we have that $d_{\rm LP}(\mathcal{D}_{A_i}(1_{\Omega_i}),\mathcal{D}_{A_i}(v_i))$ goes to $0$ as $i$ goes to infinity. It follows that $\mathcal{D}_{A_i}(1_{\Omega_i})$ weakly converges to $\mathcal{D}_A(1_\Omega)$. Since $1_{\Omega_i} A_i=c1_{\Omega_i}$, the proof is complete.
\end{proof}

\begin{lemma}\label{closedprop3} Let $A\in\mathcal{B}(\Omega)$ be a Markov graphop. Then $\|A\|_{2\to 2}=1$. 
\end{lemma}

\begin{proof}  First we show that $\|A\|_{\infty\to\infty}=1$.
Let $v\in L_{[0,1]}^\infty(\Omega)$. We have that $1_\Omega-v$ is non-negative and thus $(1_\Omega-v)A=1_\Omega-vA$ is non-negative. It follows that $vA$ is non-negative with $\|vA\|_\infty\leq 1$. Let $v\in L^\infty_{[-1,1]}(\Omega)$. We can write $v=v_1-v_2$ such that $\|v_1\|_\infty,\|v_2\|_\infty\leq 1$ and both $v_1,v_2$ are non-negative. We have shown that the values of $v_1A$ and $v_2A$ are in $[0,1]$ and so $vA=v_1A-v_2A$ takes values in $[-1,1]$. It follows that $\|vA\|_\infty\leq 1$. In general, we have for $v\in L^\infty(\Omega)$ with $m:=\|v\|_\infty$ that $v\in L^\infty_{[-m,m]}$ and thus by linearity and the previous statement we obtain that $\|vA\|_\infty\leq m.$

The fact that $1_\Omega A=1_\Omega$ implies that $\|A\|_{2\to 2}\geq 1$. Now let $v\in L^\infty(\Omega)$ be arbitrary. Then for every $k\in\mathbb{N}$ we have that $\|vA^k\|_2\leq\|vA^k\|_\infty\leq\|v\|_\infty$. By the spectral theorem this is  possible only if $\|A\|_{2\to 2}\leq 1$.
\end{proof}

\begin{theorem}\label{markcomp} Let $\mathcal{M}$ be the set of weak equivalence classes of Markov graphops. Then $(\mathcal{M},d_M)$ is a compact metric space. 
\end{theorem}

\begin{proof} Lemma \ref{closedprop3} together with Proposition \ref{closedprop} and Proposition \ref{closedprop2} implies that if $\{A_i\}_{i=1}^\infty$ is a sequence of Markov graphops with limit $A$, then $A$ is also a Markov graphop. The compactness follows from Lemma \ref{lem:sc} and Theorem \ref{exlim}.
\end{proof}

\section{Construction of the limit object}\label{proofexlim}

In this section we prove Theorem \ref{exlim}. 
Let $\{(\Omega_i,\mathcal{A}_i,\mu_i)\}_{i=1}^\infty$ be a sequence of measure spaces. Assume that $\{A_i\in\mathcal{B}_{p,q}(\Omega_i)\}_{i=1}^\infty$ is a convergent sequence of $P$-operators such that $\sup_{i\in\mathbb{N}}\|A_i\|_{p\rightarrow q}\leq c$ for some $c\in\mathbb{R}^+$. For every $k\in\mathbb N$ we define \[X_k:=\lim_{i\rightarrow\infty} \mathcal{S}^*_k(A_i).\] 
We wish to prove that there is a $P$-operator $A\in\mathcal{B}_{p,q}(\Omega)$ for some probability space $(\Omega,\mathcal{A},\mu)$ such that for every $k\in\mathbb{N}$ we have that $$\lim_{i\to\infty}\mathcal{S}^*_k(A_i)=\mathcal{S}_k^*(A).$$ 
We will need the next algebraic notion.

\begin{definition}\label{fsg} {\bf (Free  semigroup with operators)} Let $G$ and $L$ be sets. We denote by $F(G,L)$ the free semigroup with generator set $G$ and operator set $L$ (freely acting on $F(G,L)$).  More precisely, we have that $F(G,L)$ is the smallest set of abstract words satisfying the following properties.
\begin{enumerate} 
\item $G\subseteq F(G,L)$.
\item If $w_1,w_2\in F(G,L)$, then $w_1w_2\in F(G,L)$.
\item If $w\in F(G,L), l\in L$, then $l(w)\in F(G,L)$.
\end{enumerate}
There is a unique length function $m: F(G,L)\to\mathbb{N}$ such that $m(g)=1$ for $g\in G$, $m(w_1w_2)=m(w_1)+m(w_2)$ and $m(l(w))=m(w)+1$. 
\end{definition}

An example for a word in $F(G,L)$ is $l_3(l_1(g_1g_2l_2(g_2)g_3))l_3(g_2)g_1$, where $g_1,g_2,g_3\in G$ and $l_1,l_2,l_3\in L$. Note that if both $G$ and $L$ are countable sets, then so is $F(G,L)$.

\noindent{\bf Construction of a function system.}~In this technical part of the proof we construct a function system $\{v_{i,f}\in L^\infty(\Omega_i)\}_{i\in\mathbb{N},f\in F}$ for some countable index set $F$. Later we will use this function system to construct a probability distribution $\kappa\in\mathcal{P}(\mathbb{R}^{F\times\{0,1\}})$ and an operator $A\in \mathcal B_{p,q}(\mathbb{R}^{F\times\{0,1\}},\kappa)$. We will show that $A$ is an appropriate limit object for the sequence $\{A_i\}_{i=1}^\infty$. 

First we describe the index set $F$. For every $k\in\mathbb{N}$ let $X_k'\subseteq X_k$ be a countable dense subset in the metric space $(X_k,d_{\rm LP})$. Let $G:=\bigcup_{k=1}^\infty X_k'\times [k]$. Let $L=\mathbb{Q}\times\mathbb{Q}^+$. Let $F:=F(G,L)$ be as in Definition \ref{fsg}. We have that $F$ is countable.

Now we describe the functions $\{v_{i,g}\}_{i\in\mathbb{N},g\in G}$. For every $i,k\in\mathbb{N}$ and $t\in X_k'$ let $\{v_{i,(t,j)}\}_{j=1}^k$ be a system of functions in $L^\infty_{[-1,1]}(\Omega_i)$ such that the joint distribution of 
$$(v_{i,(t,1)},v_{i,(t,2)},\dots,v_{i,(t,k)},v_{i,(t,1)}A_i,v_{i,(t,2)}A_i,\dots,v_{i,(t,k)}A_i)$$
converges to $t$ as $i$ goes to infinity. 

To continue with our construction we need to interpret elements of $L$ as non-linear operators on function spaces. For $y\in\mathbb{Q}$ and $z\in\mathbb{Q}^+$ let $h_{y,z}:\mathbb{R}\to\mathbb{R}$ be the bounded, continuous function defined by $h_{y,z}(x)=0$ if $x\notin (y-z,y+z)$ and $h_{y,z}(x)=1-|x-y|/z$ if $x\in (y-z,y+z)$. For every $i\in\mathbb{N}, l\in L$ and $v\in L^{\infty}(\Omega_i)$ we define $l(v):=h_{y,z}\circ (v A_i)$, where $l$ is given by the pair $(y,z)\in\mathbb{Q}\times\mathbb{Q}^+$. Note that by definition $\|l(v)\|_\infty\leq 1$.

Now for every $i\in\mathbb{N}$ we construct the functions $\{v_{i,w}\}_{i\in\mathbb{N},w\in F}$ recursively to the length of $m(w)$. For words of length $1$ the functions are already constructed above. Assume that for some $1\leq k\in\mathbb{N}$ we have already constructed all the functions $v_{i,w}$ with $m(w)\leq k$. Let $w\in F$ such that $m(w)=k+1$. If $w=w_1w_2$ for some $w_1,w_2\in F$, then $v_{i,w}:=v_{i,w_1}v_{i,w_2}$. If $w=l(w_1)$, then $v_{i,w}:=l(v_{i,w_1})$.

\noindent {\bf Construction of the probability space.}~Let $\xi_i:\Omega_i\to\mathbb{R}^{F\times\{0,1\}}$ be the function such that for $f\in F,e\in\{0,1\}$ and $\omega_i\in\Omega_i$ the $(f,e)$ coordinate of $\xi_i(\omega_i)$ is equal to $(v_{i,f}A_i^e)(\omega_i)$, where $A_i^0$ is defined to be the identity operator. Let $\kappa_i\in\mathcal{P}(\mathbb{R}^{F\times\{0,1\}})$ denote the distribution of the random variable $\xi_i$. (In other words $\kappa_i$ is the joint distribution of the functions $\{v_{i,f}\}_{f\in F}$ and $\{v_{i,f}A_i\}_{f\in F}$.) Since $\tau(\kappa_i)\leq c$ holds for every $i$ (recall equation \eqref{eq:tau} for the definition of $\tau$), we have that there is a strictly growing sequence $\{n_i\}_{i=1}^\infty$ of natural numbers such that $\kappa_{n_i}$ is weakly convergent with limit $\kappa$ as $i$ goes to infinity. Let $\Omega:=\mathbb{R}^{F\times\{0,1\}}$ be the probability space with the Borel $\sigma$-algebra $\mathcal{A}$ and measure $\kappa$.

\noindent{\bf Construction of the operator.}~ We will define an operator $A\in \mathcal B_{p,q}(\Omega)$ with $\Omega$ defined above.  For $(f,e)\in F\times\{0,1\}$ let $\pi_{(f,e)}:\mathbb{R}^{F\times\{0,1\}}\to\mathbb{R}$ denote projection function to the coordinate at $(f,e)$. Notice that 
\begin{equation}\label{eq:pi} \pi_{(f,e)} \circ \xi_i=v_{i,f}A_i^e\qquad (i\in \mathbb N, (f, e)\in F\times \{0,1\}). \end{equation} In particular, due to the definition of $\kappa$, we also have $\pi_{(f,0)}\in L^{\infty}_{[-1,1]}(\Omega)$  for $f\in F$.  
Our goal is to show that there is a unique $(p,q)$-bounded  linear operator $A$ from $L^{\infty}(\Omega)$ to $L^{1}(\Omega)$ with $\|A\|_{p\rightarrow q}\leq c$ such that $\pi_{(f,0)}A=\pi_{(f,1)}$ holds for every $f\in F$.

\begin{lemma}\label{limconstlem1} The coordinate functions on $\mathbb{R}^{F\times\{0,1\}}$ have the following properties.
\begin{enumerate}
\item If $f_1,f_2\in F$, then $\pi_{(f_1f_2,0)}=\pi_{(f_1,0)}\pi_{(f_2,0)}$ holds in $L^{\infty}(\Omega)$.
\item If $f\in F$ and $l=(y,z)$ holds for some $y,z$, then $\pi_{(l(f),0)}=h_{y,z}\circ\pi_{(f,1)}$ holds in $L^{\infty}(\Omega)$.
\item If $a_1,a_2,\dots,a_k\in F,~\lambda_1,\lambda_2,\dots,\lambda_k\in\mathbb{R}$, then $$\Big\|\sum_{j=1}^k \lambda_j \pi_{(a_j,1)}\Big\|_q\leq c\Big\|\sum_{j=1}^k \lambda_j \pi_{(a_j,0)}\Big\|_p .$$
\item The linear span of the functions $\{\pi_{(f,0)}\}_{f\in F}$ is dense in the space $L^p(\Omega)$.

%If $v$ is an arbitrary function in $L^2(\Omega)$ then for every $\epsilon>0$ there are $a_1,a_2,\dots,a_k\in F~,~\lambda_1,\lambda_2,\dots,\lambda_k\in\mathbb{R}$ such that $$\|v-\sum_{j=1}^k \lambda_j \pi_{(a_j,0)}\|_2\leq\epsilon.$$
\item Assume that $k\in\mathbb{N}$ and $t\in X_k'$. Then $(t,j)\in G\subset F$ holds for $1\leq j\leq k$ and
$$\mathcal{D}(\pi_{((t,1),0)},\pi_{((t,2),0)},\dots,\pi_{((t,k),0)},\pi_{((t,1),1)},\pi_{((t,2),1)},\dots,\pi_{((t,k),1)})=t.$$

\end{enumerate}
\end{lemma}

\begin{remarka} Note that when functions on $\Omega$ are treated as functions in $L^r(\Omega)$ for some $r\in [1,\infty]$, then they are considered to be equal if they differ on a $\kappa$ zero measure set. This kind of weak equality of functions enables various algebraic correspondences between different coordinate functions, which would be impossible in a strict sense. As a toy example let us consider the uniform measure $\mu$ on $\{(x,x):x\in [0,1]\}$ which is a Borel measure on $\mathbb{R}^2$. We have that the $x$-coordinate function $(x,y)\mapsto x$ is equal to the $y$-coordinate function $(x,y)\mapsto y$ in the space $L^r(\mathbb{R}^2,\mu)$, because they agree on the support of $\mu$.
\end{remarka}

The proof of Lemma \ref{limconstlem1} will use the following two lemmas.

\begin{lemma}\label{limconstlem2} Let $r\in [1,\infty)$. For every $v\in L^r(\Omega)$ we have that
$$\lim_{n\to \infty} \Bigl\|v-\sum_{j=-n^2}^{n^2} (j/n)h_{j/n,1/n}\circ v\Bigr\|_r=0.$$
\end{lemma}

\begin{proof} Let $s_n:=\sum_{j=-n^2}^{n^2} (j/n)h_{j/n,1/n}$. The statement is equivalent to $\lim_{n\to\infty}\|v-s_n\circ v\|_r=0$. An elementary calculation shows that $s_n(x)=x$ if $x\in [-n,n]$ and $|s_n(x)|\leq n$ holds for $x\in K_n:=\mathbb{R}\setminus [-n,n]$. Now we have that $$v-s_n\circ v=1_{K_n}\circ(v-s_n\circ v)=1_{K_n}\circ v-1_{K_n}\circ s_n\circ v,$$ and hence $$\|v-s_n\circ v\|_r\leq \|1_{K_n}\circ v\|_r+n\mathbb{P}(|v|\geq n)=\|1_{K_n}\circ v\|_r+n\mathbb{P}(|v|^r\geq n^r),$$ where $v$ is viewed as a random variable on $\Omega$. For $r=1$, it is well known that $v\in L^1(\Omega)$ implies $\lim_{n\rightarrow\infty} n \mathbb P(|v|\geq n)=0$. For $r>1$, by Markov's inequality we have that $\mathbb{P}(|v|^r\geq n^r)$ is at most $\|v\|_r^r/n^r$ and thus $n\mathbb{P}(|v|\geq n)$ converges to $0$ as $n$ goes to infinity. Now it suffices to show that $\lim_{n\to\infty}\|1_{K_n}\circ v\|_r=0$. Let $U_n:=\{x:x\in\Omega , |v(x)|>n\}$. We have that $\|1_{K_n}\circ v\|^r_r=\int_{U_n} |v|^r d\kappa$.  The fact that $v$ is measurable implies that $\lim_{n\to\infty} \kappa(U_n)=0$ and $\lim_{n\to\infty}\int_{U_n} |v|^r d\kappa=0$.
\end{proof}

The following lemma is easy to prove, see e.g.\ Theorem 22.4 in the lecture notes \cite{lemma45}.

\begin{lemma}\label{limconstlem3} Let $r\in [1,\infty)$. Let $\{v_i\in L^\infty(\Omega)\}_{i\in I}$ be a system of functions for some countable index set $I$ such that for every $a,b\in I$ there is $c\in I$ with $v_av_b=v_c$. Let $\mathcal{A}_0$ be the $\sigma$-algebra generated by the functions $\{v_i\}_{i\in I}$. Suppose that the constant $1$ function on $\Omega$ can be approximated by a uniformly bounded family of finite linear combinations of $\{v_i\}_{i\in I}$. Then the $L^r$-closure of the linear span of $\{v_i\in L^\infty(\Omega)\}_{i\in I}$ is $L^r(\Omega,\mathcal{A}_0,\kappa)$.
\end{lemma}

Now we return to the proof of Lemma \ref{limconstlem1}.

\begin{proof}
To prove the first statement of the lemma recall that by the construction of the function system we have for every $i\in\mathbb{N}$ and $f_1,f_2\in F$ that $v_{i,f_1f_2}=v_{i,f_1}v_{i,f_2}$ holds.  Therefore, by equation \eqref{eq:pi}, we have that each $\kappa_i$  is supported inside the closed set $$\big\{\omega: \omega\in\mathbb{R}^{F\times\{0,1\}},~\pi_{(f_1f_2,0)}(\omega)=\pi_{(f_1,0)}(\omega)\pi_{(f_2,0)}(\omega)\big\}.$$ This implies that $\kappa$ is also supported inside this set and thus $\pi_{(f_1f_2,0)}=\pi_{(f_1,0)}\pi_{(f_2,0)}$ holds $\kappa$-almost everywhere. 

The proof of the second statement is similar to the first one. Again by the construction of the function system we have for every $i\in\mathbb{N}$ and $f\in F,l=(p,q)\in L$ that $v_{i,l(f)}=l(v_{i,f})=h_{p,q}\circ (v_{i,f} A_i)$. This means by the definition of $\kappa_i$ and equation \eqref{eq:pi} that $\kappa_i$ is supported on the closed set
$$\big\{\omega: \omega\in\mathbb{R}^{F\times\{0,1\}},~\pi_{(l(f),0)}(\omega)=h_{p,q}(\pi_{(f,1)}(\omega))\big\}$$ for every $i\in\mathbb{N}$. Thus $\pi_{(l(f),0)}=h_{p,q}\circ\pi_{(f,1)}$ holds $\kappa$-almost everywhere.

For the proof of the third statement, let us use that $\|A_i\|_{p\to q}\leq c$ holds for every $i\in\mathbb{N}$ and hence
$$\Big\|\sum_{j=1}^k \lambda_j v_{i,a_j}A_i\Big\|_q\leq c\Big\|\sum_{j=1}^k \lambda_j v_{i,a_j}\Big\|_p .$$ 
Since the sum on the right-hand side is a function in $L^\infty(\Omega_i)$ whose values are in the compact interval $[-\lambda,\lambda]$ for $\lambda:=\sum_{j=1}^k |\lambda_j|$,  we have that $\sum_{j=1}^k \lambda_j \pi_{(a_j,0)}$ is a bounded, continuous function on the support of $\kappa$. Therefore, using $\kappa_{n_i}\stackrel{w}\rightarrow \kappa$ and  equation \eqref{eq:pi} again (in particular, integrating of the $p$th power of the absolute values with respect to $\kappa_i$), we obtain that  
$$\lim_{i\to\infty}\Big\|\sum_{j=1}^k \lambda_j v_{n_i,a_j}\Big\|_p=\Big\|\sum_{j=1}^k \lambda_j \pi_{(a_j,0)}\Big\|_p.$$
On the other hand, $\big|\sum_{j=1}^k \lambda_j \pi_{(a_j,1)}|^q$ is a nonnegative continuous function, thus weak convergence in this case implies the following inequality: 
$$\Big\|\sum_{j=1}^k \lambda_j \pi_{(a_j,1)}\Big\|_q\leq\limsup_{i\to\infty}\Big\|\sum_{j=1}^k \lambda_j v_{n_i,a_j}A_i\Big\|_q.$$ These inequalities together yield the third statement.

For the fourth statement, let $\mathcal{H}_r$ denote the $L^r$-closure of the linear span of the functions $\{\pi_{(f,0)}\}_{f\in F}$ for $r\in [1,\infty)$. First we show that $\pi_{(f,1)}\in \mathcal{H}_q$ holds for every $f\in F$. By the second statement we have that
\begin{equation}\label{limconsteq1}
\sum_{j=-n^2}^{n^2} (j/n)h_{j/n,1/n}\circ \pi_{(f,1)}=\sum_{j=-n^2}^{n^2} (j/n) \pi_{(l_j(f),0)},
\end{equation}
 where $l_j$ is given by the pair $(j/n,1/n)$ for $-n^2\leq j\leq n^2$. Since the right-hand side of (\ref{limconsteq1}) is in $\mathcal{H}_q$, we obtain that the left-hand side is also in $\mathcal{H}_q$. On the other hand, we have $\pi_{(f,1)}\in L^q(\Omega)$ due to the third statement.  Hence by Lemma \ref{limconstlem2} we obtain that, as $n$ goes to infinity, the left-hand side of (\ref{limconsteq1}) converges to $\pi_{(f,1)}$ in $L^q(\Omega)$ and thus $\pi_{(f,1)}\in \mathcal{H}_q$.
 
Let $\mathcal{A}_0$ be the $\sigma$-algebra generated by the functions $\{\pi_{(f,0)}\}_{f\in F}$. Notice that the constant $1$ function on $\Omega$ can be approximated already in $X_1'$. We have by the first statement in this lemma and Lemma \ref{limconstlem3} that $\mathcal{H}_r=L^r(\Omega,\mathcal{A}_0,\kappa)$ holds for every $r\in [1,\infty)$. As we have shown, we have for every $f\in F$ that $\pi_{(f,1)}\in \mathcal{H}_q=L^q(\Omega,\mathcal{A}_0,\kappa)$ and thus all coordinate functions on $\mathbb{R}^{F\times\{0,1\}}$ are measurable in $\mathcal{A}_0$. This shows that $\mathcal{H}_r=L^r(\Omega,\mathcal{A}_0,\kappa)=L^r(\Omega,\mathcal{A},\kappa)=L^r(\Omega)$ holds for every $r\in [1,\infty)$.

The last statement of the lemma follows directly from the definition of the functions $\{v_{i,(t,j)}\}_{i\in\mathbb{N},j\in [k]}$ and the definition of $\kappa$.
\end{proof}

We are ready to define the operator $A\in\mathcal{B}_{p,q}(\Omega)$. For $f\in F$, let $\pi_{(f,0)}A=\pi_{(f,1)}$. This defines a linear operator on the linear span of $\{\pi_{(f,0)}\}_{f\in F}$, which is bounded due to the third statement of Lemma \ref{limconstlem1}. Hence it has a unique continuous linear extension on its $L^p$-closure. By the fourth statement of the same lemma we get that there is a unique operator $A\in\mathcal{B}_{p,q}(\Omega)$ with $\|A\|_{p\to q}\leq c$ such that $\pi_{(f,0)}A=\pi_{(f,1)}$ holds for every $f\in F$.

\noindent{\bf Last part of the proof.}~The last statement of Lemma \ref{limconstlem1} together with the equality $\pi_{((t,j),0)}A=\pi_{((t,j),1)}$ imply that for every $k\in\mathbb{N}$ and $t\in X_k'$ we have  $t\in\mathcal{S}_k(A)$. Therefore for every $k\in\mathbb{N}$ we have that $X_k\subseteq \mathcal{S}_k^*(A)$. Our goal is to show that $X_k=\mathcal{S}_k^*(A)$ for every $k\in\mathbb{N}$ and thus it remains to prove that $\mathcal{S}_k^*(A)\subseteq X_k$.

Let $k\in\mathbb{N}$ and let $v_1,v_2,\dots,v_k\in L^\infty_{[-1,1]}(\Omega)$. We need to show that $\alpha:=\mathcal{D}_A(\{v_j\}_{j=1}^k)$ is in  $X_k$. Let $\varepsilon>0$ be arbitrary. We have by the fourth statement of Lemma \ref{limconstlem1} that for some large enough natural number $m$ there are elements $f_1, f_2, \dots, f_m\in F$ and real numbers $\{\lambda_{a,j}\}_{a\in [m],j\in [k]}$ such that for every $j\in [k]$ we have $\|w_j-v_j\|_p\leq\varepsilon$, where $w_j:=\sum_{a=1}^m \lambda_{a,j}\pi_{(f_a,0)}$ for $j\in [k]$.

Since only vectors with infinity norm at most $1$ are considered in the profile, we will use a truncating function. Namely, let $\tilde h:\mathbb{R}\to\mathbb{R}$ be the continuous function with $\tilde h(x)=x$ for $x\in[-1,1]$, $\tilde h(x)=-1$ for $x\in (-\infty,-1]$ and $\tilde h(x)=1$ for $x\in [1,\infty)$. Since $\|v_j\|_\infty\leq 1$ holds for $j\in [k]$ we have that $|w_j(x)-v_j(x)|\geq |\tilde h\circ w_j(x)-v_j(x)|$ holds almost everywhere and thus by $\|w_j-v_j\|_p\leq\varepsilon$ we obtain $\|\tilde h\circ w_j-v_j\|_p\leq\varepsilon$ for $j\in[k]$. This also implies by the triangle inequality that 
\begin{equation}\label{limconsteq3}
\|\tilde h\circ w_j-w_j\|_p\leq\|\tilde h\circ w_j-v_j\|_p+\|v_j-w_j\|_p\leq 2\varepsilon
\end{equation}
holds for $j\in [k]$.

For $i\in\mathbb{N}$ and $j\in[k]$ let $z_{i,j}:=\sum_{a=1}^m \lambda_{a,j}v_{i,f_a}$. Let $\beta_i:=\mathcal{D}_{A_i}(\{z_{i,j}\}_{j=1}^k)$. By the definition of $\kappa$ we have that  $$\beta:=\lim_{i\to\infty}\beta_{n_i}=\mathcal{D}_A(\{w_j\}_{j=1}^k)$$ holds in $d_{\rm LP}$. Since  $w_j\in L^{\infty}(\Omega)$, we have $\|v_jA-w_jA\|_1\leq \|v_jA-w_jA\|_q\leq c\|v_j-w_j\|_p\leq c\varepsilon$. We get from Lemma \ref{coupdist2} that $d_{\rm LP}(\alpha,\beta)\leq (2k)^{3/4}(c'\varepsilon)^{1/2}$, where $c':=\max(c,1)$.

Let 
$$\beta'_i:=\mathcal{D}_{A_i}\big(\{\tilde h\circ z_{i,j}\}_{j=1}^k).$$
Notice that the functions $\tilde h\circ z_{i,j}-z_{i,j}$ are uniformly bounded and their distribution converges weakly to the distribution of $\tilde h\circ w_j-w_j$. Hence if $i$ is large enough, then by (\ref{limconsteq3}) we have that $\|\tilde h\circ z_{i,j}-z_{i,j}\|_p\leq 3\varepsilon$ holds for $j\in [k]$  and thus $d_{\rm LP}(\beta'_i,\beta_i)\leq (2k)^{3/4}(3c'\varepsilon)^{1/2}$ by Lemma \ref{coupdist2}. 

Let $\{n_i'\}_{i=1}^\infty$ be a subsequence of $\{n_i\}_{i=1}^\infty$ such that $\beta':=\lim_{i\to\infty}\beta'_{n'_i}$ exists. Observe that $\beta'\in X_k$ and $d_{\rm LP}(\beta',\beta)\leq (2k)^{3/4}(3c'\varepsilon)^{1/2}$. We obtain that $$d_{\rm LP}(X_k,\alpha)\leq d_{\rm LP}(\beta',\alpha)\leq d_{\rm LP}(\beta',\beta)+d_{\rm LP}(\beta,\alpha)\leq 3(2k)^{3/4}(c'\varepsilon)^{1/2}.$$ This holds for every $\varepsilon>0$ and thus $\alpha\in X_k$. \hfill  $\square$

\section{General graph limits}\label{chapgenlim}

\label{general}

There are various ways of representing graphs by operators and in particular by $P$-operators. Depending on the chosen representation we get a corresponding limit notion for graphs. In this section we list four natural operator representations of graphs and investigate the corresponding graph limit notions. Let $G$ be a finite graph on the vertex set $V(G)=[n]$ with edge set $E(G)$. 

\noindent{\bf Adjacency operator convergence:}~ Recall that $\mu_{[n]}$ denotes the uniform distribution on $[n]$.  We denote by $A(G)\in\mathcal{B}([n],\mu_{[n]})$ the $P$-operator defined by 
$$(vA(G))(i):=\sum_{(j,i)\in E(G)} v(j)$$ for $i\in [n]$. We have that $d\geq\|A(G)\|_{2\to 2}\geq d^{1/2}$, where $d$ is the maximal degree in $G$. We can say that a graph sequence $\{G_i\}_{i=1}^\infty$ is convergent if $\{A(G_i)\}_{i=1}^\infty$ is an action convergent sequence of $P$-operators. We obtain compactness for graphs with uniformly bounded degree. Quite surprisingly (and non-trivially) it turns out that this convergence notion is equivalent to local-global convergence, which is a refinement of Benjamini--Schramm convergence (see \cite{BS, HLSz}). However, for graphs with non-bounded degrees compactness is not guaranteed. This can be solved by scaling the operators $A(G)$ by some number that depends on $G$. For example we have that $\|A(G)/|V(G)|\|_{2\to 2}\leq 1$ holds for every graph $G$. Again quite surprisingly it turns out that convergence of $A(G_i)/|V(G_i)|$ is equivalent to dense graph convergence. (For a definition of dense graph convergence see \cite{LSz}.) This motivates us to introduce scaling functions that map graphs to positive real numbers. Let $\mathcal{G}$ denote the set of isomorphism classes of finite graphs. 

\begin{definition} Let $f:\mathcal{G}\to\mathbb{R}^+$ be a function. We say that a graph sequence $\{G_i\}_{i=1}^\infty$ is adjacency operator convergent (or just convergent) with scaling $f$ if the sequence $\{A(G_i)/f(G_i)\}_{i=1}^\infty$ is an action convergent sequence of $P$-operators. 
\end{definition}

Recall that a {\it graphop} is a self-adjoint, positivity-preserving $P$-operator. The word graphop is a mixture of the words graph and operator. Note that both graphons and graphings used in graph limit theory are graphops. Theorem \ref{exlim} implies the following.

\begin{theorem} Let $\{G_i\}_{i=1}^\infty$ be an adjacency operator convergent sequence of graphs with scaling $f:\mathcal{G}\to\mathbb{R}^+$. Assume that there exist $p\in [1,\infty),q\in (1,\infty)$ and $c\in\mathbb{R}^+$ such that $\|G_i/f(G_i)\|_{p\to q}\leq c$ holds for every $i\in\mathbb{N}$. Then there is a graphop $A$ such that $\lim_{i\to\infty} A(G)/f(G_i)=A$ holds. We say that $A$ is the adjacency operator limit of $(\{G_i\}_{i=1}^\infty,f)$.
\end{theorem}

A natural scaling  is $f_{p,q}(G):=\|A(G)\|_{p\to q}$ defined for non-empty graphs $G$ ($f_{p,q}$ can be defined as $1$ on the empty graph), where $p\in [1,\infty),q\in (1,\infty]$. Let us call it {\it norm scaling}. With this scaling every graph sequence has a convergent subsequence, hence we have sequential compactness for arbitrary graph sequences. Norm scaling leads to a general convergence notion that generalizes local-global convergence and recovers dense graph limits up to a constant multiplicative factor in the limit object. The norm scaling is very convenient to use for general graph sequences where no other natural normalization is given.

\medskip

\noindent{\bf Random walk convergence of graphs:}~Let $\nu_G$ denote the stationary measure of the random walk on $G$. It is well known that  $\nu_G$ is the probability measure on $[n]$ with $\nu_G(i)=d_i/(2|E(G)|)$,  where $d_i$ is the degree of the vertex $i$ for $i\in [n]$. We denote by $M(G)\in\mathcal{B}_{2,2}([n],\nu_G)$ the $P$-operator defined by equation \eqref{eq:vmg}. The operator $M(G)$ is known as the Markov kernel corresponding to the random walk on $G$. We have that $M(G)$ is a Markov graphop. Consequently, by Lemma \ref{closedprop3} we have that $\|M(G)\|_{2\rightarrow 2}=1$. (If $G$ has no edges, then $M(G)$ is not defined.)

\begin{definition} A graph sequence of non-empty graphs $\{G_i\}$ is called  random walk convergent if $\{M(G_i)\}_{i=1}^\infty$ is a convergent sequence of $P$-operators. 
\end{definition}

The following theorem is a direct consequence of Theorem \ref{markcomp}. 

\begin{theorem} Every graph sequence $\{G_i\}_{i=1}^\infty$ has a random walk convergent subsequence. If $\{G_i\}_{i=1}^\infty$ is random walk convergent, then there is a Markov graphop $A$ such that $\lim_{i\to \infty} M(G_i)=A$. We say that $A$ is the random walk limit of $\{G_i\}_{i=1}^\infty$.
\end{theorem}

Note that for regular graphs random walk convergence is equivalent to  adjacency operator convergence with scaling by $\|G\|_{2\to 2}$. However, if there is a small but non-zero number of very high degree points in $G$ and many low degree points, then adjacency operator convergence may trivialize. Examples for this are the star graphs or the 2-subdivisions of complete graphs. In these cases random walk convergence turns out to be more natural and leads to interesting and non-trivial limit objects (see Section \ref{chapex}).

\noindent{\bf Extended random walk convergence:} A Markov pair is a pair of a Markov graphop $A\in\mathcal{B}(\Omega)$ and a measurable function $f$ on $\Omega$. As we explained in the introduction, a sequence of Markov pairs $\{(A_i,f_i)\}_{i=1}^\infty$ is convergent if the extended $k$-profiles formed by distributions of the form
$$\mathcal{D}_{A_i,f}(\{v_j\}_{j=1}^k):=\mathcal{D}(\{v_j\}_{j=1}^k,\{v_jA_i\}_{j=1}^k,f)$$ converge in $d_H$ for every $k\in\mathbb{N}$. It can be proved with a slight extension of the proof of Theorem \ref{exlim} (details will be worked out elsewhere) that a convergent sequence of Markov pairs has a limit which is also a Markov pair. There are two different uses of Markov pairs. The first one is the following. In spite of the fact that the Markov kernel of a finite graph $M(G)$ determines the sequence of the non-zero degrees (even with multiplicities), this information may be lost in the limit. Even if it is preserved in some way (examples for this are sequences of bounded degree graphs), degrees can not be read off in the usual way from the limit object which is a Markov graphop. The idea is that we can store the information on the degrees in a normalized version $d^*$ of the degree function $d$. 

\begin{remarka} {\bf (Representing graphops by Markov pairs)} In general, every graphop $A\in\mathcal{B}(\Omega,\mathcal{A},\mu)$ can be naturally represented by a Markov pair in the following way. Let $\nu$ be the representing measure of $A$ given by Theorem \ref{measrepthm} and let $\nu'$ be the marginal measure of $\nu$ on $(\Omega,\mathcal{A})$. Let $M(A)$ denote the Markov graphop on $(\Omega,\mathcal{A},\nu')$ determined by the measure $\nu$ using again Theorem \ref{measrepthm}. Note that the action of $M(A)$ is given by $$(vM(A))(x)=(vA)(x)/(1_\Omega A)(x).$$ The representation of $A$ is given by the pair $(M(A),1_\Omega A)$.
\end{remarka}

Another interesting use of Markov pairs is that we can use them to represent generalized graphons (related to graphexes \cite{BCCH, janson}) that are symmetric non-negative measurable functions of the form $W:\mathbb{R}^+\times\mathbb{R}^+\to [0,1]$ with $\|W\|_1<\infty$. Let $\nu$ denote the probability measure on $\mathbb{R}^+\times\mathbb{R}^+$ defined by $\nu(S)=\|W\|_1^{-1}\int_S W~d\lambda^2$. Let $\nu'$ be the marginal distribution of $\nu$ on $\mathbb{R}^+$. Let $M$ denote the Markov graphop on $(\mathbb{R}^+,\nu')$ determined by Theorem \ref{measrepthm}. For $x\in\mathbb{R}^+$ let $f(x):=\int_{\mathbb{R}^+} W(x,y)~d\lambda$. We can represent the generalized graphon $W$ by the Markov pair $(M,f)$.

\medskip

\noindent{\bf Laplace operator convergence:}~Using the above notations we denote by $L(G)\in\mathcal{B}_2([n],\mu_{[n]})$ the $P$-operator defined by 
$$(vL(G))(i):=d(i)v(i)-\sum_{(j,i)\in E(G)} v(j).$$
We have that $L(G)$ is a positive self-adjoint operator. Note that in contrast with $A(G)$ and $M(G)$, the operator $L(G)$ is typically not positivity-preserving. On the other hand we gain positiveness. Similarly to the previous definitions we say that a graph sequence $\{G_i\}_{i=1}^\infty$ is Laplace operator convergent (or just convergent) with scaling $f$ if the sequence $\{L(G_i)/f(G_i)\}_{i=1}^\infty$ has uniformly bounded operator norm and is a convergent sequence of $P$-operators. Limit objects are positive, self-adjoint $P$-operators. 

\medskip

\noindent{\bf Degree weighted operator convergence:}~Finally we mention one more interesting $P$-operator related to $G$. Similarly to $M(G)$ we use the stationary distribution of the random walk. Let $F(G)\in\mathcal{B}([n],\nu_G)$ be defined by $$(vF(G))(i):=\sum_{(j,i)\in E(G)} v(j)d(j).$$ Again we have that $F(G)$ is a graphop. Indeed, positivity-preserving property is clear, and self-adjointness can be verified by a simple calculation. Limits of appropriately normalized versions of $F(G_i)$ are graphops. It may be interesting to mention that this concept resonates with the so-called PageRank algorithm \cite{pagerank} in the sense that the operator puts larger weight on the neighbors with higher degree when calculating the image of a vector at a given vertex.

\section{Measure representation of graphops}

Let $A\in\mathcal{B}(\Omega,\mathcal{A},\mu)$ be a graphop. In this section we construct a measure $\nu$ on $\Omega\times\Omega$ that represents the operator $A$. This means that the operator $A$ can be reconstructed from the measure $\nu$ in a natural way. Intuitively, if we think of $A$ as an infinite graph-like object, then $\nu$ shows where we can find the edges in $\Omega\times\Omega$. Note that both graphons and graphings are given in terms of such measures rather than in the form of operators. More precisely, graphons are given by a measurable function which is the Radon--Nykodim derivative of a measure on $[0,1]\times [0,1]$. Our goal in this chapter is to bring closer the operator language and the existing representations of graph limits. 

Assume that $(\Omega,\mathcal{A})$ is a standard Borel space. Let $\mathcal{R}$ denote the set of product sets of the form $S\times T\subseteq\Omega\times\Omega$ where $S,T\in\mathcal{A}$. We have that $\mathcal{R}$ is a so-called semiring. We define the function $\nu$ on $\mathcal{R}$ such that $\nu(S\times T):=(1_S,1_T)_A$ holds for $S,T\in\mathcal{A}$ (recall equation \eqref{skaleq}). 

\begin{lemma}\label{measreplem} The function $\nu$ has the following properties.
\begin{enumerate}
\item $\nu(S\times T)\geq 0$ for every $S,T\in\mathcal{A}$.
\item If $S\times T$ is the disjoint union of the finitely many product sets $\{S_i\times T_i\}_{i=1}^n$, then $\nu(S\times T)=\sum_{i=1}^n\nu(S_i\times T_i)$. 
\item For every $\varepsilon>0$ there is $\delta=\delta(\varepsilon)>0$ such that $\nu(S\times T)\leq \varepsilon$ whenever the minimum of $\mu(S)$ and $\mu(T)$ is at most $\delta$.
\end{enumerate}
\end{lemma}

\begin{proof} By the positivity-preserving property of $A$ and the bilinearity of $(.,.)_A$ we have that $\nu$ satisfies the first two properties. To show the last property observe that by the self-adjoint property $A$ we have that $\nu(S\times T)=\nu(T\times S)$ and so the statement is equivalent to showing the existence of $\delta>0$ such that  $\nu(S\times T)\leq\varepsilon$ whenever $\nu(T)\leq\delta$. We have by the first two properties that
 $$\nu(S\times T)\leq \nu(S\times T)+\nu((\Omega\setminus S)\times T)=\nu(\Omega\times T)=\int_T f~d\mu,$$ where $f:=1_\Omega A$. Now, since $f\geq 0$ and $\int f~d\mu<\infty$, the statement of the lemma follows from the well know absolute continuity property of integration.
\end{proof}

The proof of the following lemma follows a similar scheme as \cite[Lemma A.10]{pablo} or \cite[Theorem 454D]{fremlin}. 

\begin{lemma}\label{measreplem2} The function $\nu$ is a premeasure on $\mathcal{S}$ and it has a unique extension to a Borel measure on $(\Omega\times\Omega,\mathcal{A}\otimes\mathcal{A})$, denoted also by $\nu$.
\end{lemma}

\begin{proof}  First we claim that if $\varepsilon>0$ and $\delta>0$ satisfy the third property in lemma \ref{measreplem} and $\mu(S),\mu(T)\geq 1-\delta$, then 
\begin{equation}\label{measrepeq1}
\nu(S\times T) \geq \nu(\Omega\times\Omega)-2\varepsilon.
\end{equation}
Indeed, by Lemma \ref{measreplem} we have
$$\nu(S\times T)=\nu(\Omega\times\Omega)-\nu((\Omega\setminus S)\times T)-\nu(S\times(\Omega\setminus T))+\nu((\Omega\setminus S)\times(\Omega\setminus T))\geq$$
$$\geq \nu(\Omega\times \Omega)-\nu((\Omega\setminus S)\times\Omega)-\nu(\Omega\times(\Omega\setminus T))\geq \nu(\Omega\times\Omega)-2\varepsilon.$$

To prove that $\nu$ is a premeasure the only non-trivial part is to show that if $R\in\mathcal{R}$ is the pairwise disjoint union of sets $\{R_i=S_i\times T_i\}_{i=1}^\infty$ in $\mathcal{R}$, then $\nu(R)\leq\sum_{i=1}^\infty\nu(R_i)$. Note that it suffices to prove it for $R=\Omega\times\Omega$, since given any other product set $R$ we can obtain $\Omega\times\Omega$ as the disjoint union of $R$ and a finite number of other product sets, in such a way that the claim for $\Omega\times\Omega$ implies the claim for $R$. Now since $\Omega$ is standard, there exists a topology $\tau$ on $\Omega$ generating $\mathcal{A}$ such that 1) we can approximate every set in $\mathcal{A}$ by a compact set with arbitrary precision measured in $\mu$; 2) both $S_i$ and $T_i$ are open for every $i\in\mathbb{N}$. Let $\varepsilon>0$ arbitrary and let $K_1,K_2\in\mathcal{A}$ be $\tau$- compact sets such that $\mu(K_1),\mu(K_2)\geq 1-\delta(\varepsilon)$. We have by equation (\ref{measrepeq1}) that $\nu(K_1\times K_2)\geq \nu(\Omega\times\Omega)-2\varepsilon$. Since every $S_i$ and $T_i$ is open, we have that $\{S_i\times T_i\}_{i=1}^\infty$ is an open cover of the compact set $K_1\times K_2$, and so there is a finite sub-cover. Applying the second property of $\nu$ to this finite sub-cover we obtain that 
$$\nu(\Omega\times\Omega)\leq \nu(K_1\times K_2)+2\varepsilon\leq 2\varepsilon+\sum_{i=1}^\infty \nu(R_i).$$ Since this is true for every $\varepsilon>0$, we obtain that $\nu$ is a premeasure. The Carath\'eodory extension theorem implies that there is a unique extension of $\nu$ to $\mathcal{A}\otimes\mathcal{A}$.
\end{proof}

\begin{theorem}\label{measrepthm} {\bf (Measure representation of graphops)} If $A\in\mathcal{B}(\Omega,\mathcal{A},\mu)$ is a graphop, then there is a unique finite measure $\nu$ on $(\Omega\times\Omega,\mathcal{A}\otimes\mathcal{A})$ with the following properties.
\begin{enumerate} 
\item $\nu$ is symmetric, i.e. $\nu(S\times T)=\nu(T\times S)$ holds for every $S,T\in\mathcal{A}$.
\item The marginal distribution of $\nu$ on $\Omega$ is absolutely continuous with respect to $\mu$. 
\item 
$(f,g)_A=\int_{(x,y)\in\Omega^2}f(x)g(y)~d\nu$ holds for every $f,g\in L^\infty(\Omega)$. 
\end{enumerate}
Conversely, if $\nu$ is a finite measure on $(\Omega\times\Omega,\mathcal{A}\otimes\mathcal{A})$ satisfying the first two properties, then there is a unique graphop $A$ such that the third property is satisfied. 
\end{theorem}

\begin{proof} The existence and uniqueness of $\nu$ follows from Lemma \ref{measreplem2}. For the converse statement let $f\in L^\infty(\Omega)$ be an arbitrary function. Let $F_f$ denote the   $L^\infty(\Omega)\to\mathbb{R}$ functional defined by $F_f(g):=\int_{(x,y)\in\Omega^2}f(x)g(y)~d\nu$. We have that $|F_f(g)|\leq \|f\|_\infty\|g\|_\infty\nu(\Omega^2)$, and thus by duality there is a unique function $m(f)\in L^1(\Omega)$ such that $F_f(g)=\int m(f)(x)g(x)~d\mu$. It is clear from the definition that $A:f\mapsto m(f)$ is a self-adjoint, positivity-preserving linear operator with $\|A\|_{\infty\to 1}=\nu(\Omega^2)$ satisfying the third property.
\end{proof}

\begin{remarka} {\bf (Fiber measures)} A natural way of reconstructing $A$ from the representing measure $\nu$ goes by disintegrating the measure $\nu$. By using the disintegration theorem one obtains a family of measures $\{\nu_x\}_{x\in\Omega}$ on $(\Omega,\mathcal{A})$ (called fiber measures) such that $$(fA)(x)=\int_\Omega f~d\nu_x.$$ In general it is very convenient to describe a graphop in terms of fiber measures. This is illustrated on Figure \ref{sphericgraph}.
\end{remarka}

\begin{remarka} {\bf (Markov graphops as couplings)} Markov graphops are special graphops such that $1_\Omega A=1_\Omega$. It follows that the marginal distribution of $\nu$ on $\Omega$ is equal to $\mu$. This means that Markov graphops are completely specified by the data $(\Omega,\mathcal{A},\nu)$, where $\nu$ is a symmetric probability measure on $(\Omega\times\Omega,\mathcal{A}\otimes\mathcal{A})$. Such objects are symmetric self-couplings of probability spaces.
\end{remarka}

\section{Quotient convergence and partitions}

In the first part of this section we relate $P$-operator convergence to the so-called quotient convergence, which  was studied  in different forms by different authors \cite{BCCZ2, BR, kunszenti2}. The version that we generalize to $P$-operators was defined in \cite{kunszenti2}. In the second part of the chapter we describe a variant of action convergence that turns out to be equivalent to the original version for uniformly $(p,q)$-bounded sequences. 

\begin{definition} A function partition of $(\Omega,\mathcal{A},\mu)$ is a set $\{v_i\}_{i=1}^k$ of $0-1$ valued measurable functions on $\Omega$ such that $\sum_{i=1}^k v_i=1_\Omega$. 
A fractional function partition is a set $\{v_i\}_{i=1}^k$ of functions in $L^\infty_{[0,1]}(\Omega)$ such that $\sum_{i=1}^k v_i=1_\Omega$. We say that $\{v_i\}_{i=1}^k$ is balanced if $\|v_i\|_1=1/k$ holds for every $i\in [k]$.
\end{definition}

\begin{definition} {\bf (Quotients of $P$-operators)} Let $k\in\mathbb{N}$ and $A\in\mathcal{B}(\Omega)$. A balanced fractional $k\times k$ quotient of $A$ is a matrix $M\in\mathbb{R}^{k\times k}$ such that there is a balanced fractional function partition $\{v_i\}_{i=1}^k$ of $\Omega$ with $M_{i,j}=(v_i,v_j)_A$ for every $i,j\in [k]$. Let $\mathcal{Q}_{k}(A)$ denote the set of all balanced fractional quotients of $A$.
\end{definition}

Note that by linearity, the entry sum of any matrix $M\in\mathcal{Q}_k(A)$ is equal to $(1_\Omega,1_\Omega)_A$ for every $k\in\mathbb{N}$.
For two matrices $A,B\in \mathbb{R}^{[k]\times [k]}$ let $d_1(A,B):=\sum_{i,j}|A_{i,j}-B_{i,j}|$ denote the entry-wise $l_1$ distance. For two subsets $S_1,S_2\subseteq \mathbb{R}^{[k]\times [k]}$ let $d_{1,H}$ denote the corresponding Hausdorff distance. 

\begin{definition} {\bf (Quotient convergence and metric)} A sequence of $P$-operators $\{A_i\}_{i=1}^\infty$ is quotient convergent if for every $k$ we have that $\mathcal{Q}_k(A_i)$ is convergent in $d_{1,H}$. 
\end{definition}

The following proposition says that $P$-operator convergence is stronger than quotient convergence if the sequence has uniformly bounded $\|.\|_{p\to q}$ norm for some $p\in [1,\infty),q\in (1,\infty]$.

\begin{lemma} Let us fix $c\geq 1$ in $\mathbb{R}^+$ and numbers $p\in [1,\infty),q\in (1,\infty]$. For every $k\in\mathbb{N}$ and $\varepsilon>0$ there exists $\delta$ such that whenever two $P$-operators $A\in\mathcal{B}(\Omega_1),B\in\mathcal{B}(\Omega_2)$ with $\|A\|_{p\to q},\|B\|_{p\to q}\leq c$ satisfy $d_M(A,B)\leq\delta$, then $d_{1,H}(\mathcal{Q}_k(A),\mathcal{Q}_k(B))\leq\varepsilon$. 
\end{lemma}

\begin{proof} Depending on $p$ and $q$ let us choose $p'\in (p,\infty)$ and $q'\in (1,q)$ with $1/p'+1/q'=1$. Let $M\in\mathcal{Q}_k(B)$. We need to show that if $A$ and $B$ are sufficiently close in $d_M$, then there is $M'\in\mathcal{Q}_k(A)$ with $d_1(M,M')\leq\varepsilon$. The other direction (when $M\in\mathcal{Q}_k(A)$) follows by the symmetry of the argument.

Let $v_1,v_2,\dots,v_k$ be a balanced fractional function partition of $\Omega_2$ such that the corresponding quotient of $B$ is $M$. We have by the definition of $d_M$ that there are vectors $\{w_i\}_{i=1}^k$ in $L^\infty_{[-1,1]}(\Omega_1)$ such that $$d_{\rm LP}(\mathcal{D}_B(\{v_i\}_{i=1}^k)),\mathcal{D}_A(\{w_i\}_{i=1}^k))\leq 2^{k+1}d_M(A,B).$$ It is easy to see that depending on $k$ and an arbitrary constant $\varepsilon_2>0$ if $d_M(A,B)$ is small enough than there is a balanced fractional function partition $\{w'_i\}_{i=1}^k$ on $\Omega_1$ such that $\|w_i-w'_i\|_{p'}\leq\varepsilon_2$ holds for every $i\in [k]$. For such a function system we have for every $i,j\in [k]$ that 
$$|(w_i,w_j)_A-(w_i',w'_j)_A|\leq |(w_i,w_j)_A-(w_i',w_j)_A|+|(w_i',w_j)_A-(w_i',w_j')_A|=$$ $$=|(w_i-w_i',w_j)_A|+|(w_i',w_j-w_j')_A|\leq 2\varepsilon_2c,$$ where the last inequality is by (\ref{skalineq}). Let $M'\in\mathcal{Q}_k(A)$ be the quotient matrix of $A$ corresponding to $\{w_i'\}_{i=1}^k$ and let $M''$ be defined by $M''_{i,j}:=(w_i,w_j)_A$. We have that $d_1(M',M'')\leq 2\varepsilon_2k^2c\leq\varepsilon/2$ for appropriate $\varepsilon_2$. It remains to show that $|(v_i,v_j)_A-(w_i,w_j)_A|$ is small for every $i,j\in [k]$ if $d_M(A,B)$ is sufficiently small. This follows from Lemma \ref{closedlem2}.
\end{proof}

The following proposition is a direct consequence of the previous lemma.

\begin{proposition}\label{quprop}{\bf (Action convergence $\Rightarrow$ quotient convergence)} Let $p\in [1,\infty),q\in (1,\infty]$ and let $\{A_i\}_{i=1}^\infty$ be an action convergent sequence of  $P$-operators with uniformly bounded $\|.\|_{p\to q}$ norms. Then $\{A_i\}_{i=1}^\infty$ is quotient convergent. 
\end{proposition}

In the rest of this section we will formulate a version of action convergence. It is clear that $\{v_i\}_{i=1}^k$ is a function partition if and only if there is measurable partition $\mathcal{P}=\{P_1,P_2,\dots,P_k\}$ of $\Omega$ such that $v_i=1_{P_i}$. Let $M_k$ denote the set of probability measures $\mu$ on $\mathbb{R}^{2k}$ such that $\mu$ is concentrated on $\{e_i\}_{i=1}^k\times\mathbb{R}^k$, where $e_i\in\mathbb{R}^k$ is the vector with $1$ at the $i$-th coordinate and $0$ everywhere else. Let $A\in\mathcal{B}(\Omega,\mathcal{B},\mu)$. We have that $\{v_i\}_{i=1}^k$ is a function partition if and only $\mathcal{D}_A(\{v_i\}_{i=1}^k)\in M_k$. Let $\mathcal{S}'_k(A):=\mathcal{S}_k(A)\cap M_k$. In other words $\mathcal{S}_k'(A)$ is the set of all probability measures $\mathcal{D}_A(\{v_i\}_{i=1}^k)$, where $\{v_i\}_{i=1}^\infty$ is a function partition.
The next theorem gives a useful equivalent formulation of $P$-operator convergence for uniformly bounded operators.

\begin{theorem}\label{partthm} Let $p\in [1,\infty),q\in (1,\infty]$ and let $\{A_i\}_{i=1}^\infty$ be a uniformly $(p,q)$-bounded sequence of $P$-operators. Then $\{A_i\}_{i=1}^\infty$ is convergent if and only if for every $k$ the sequence $\mathcal{S}'_k(A_i)$ is convergent.
\end{theorem}

The proof of this theorem is a direct consequence of the following two lemmas.

\begin{lemma}\label{partlem} Let $p\in [1,\infty),q\in (1,\infty], c\in\mathbb{R}^+$ be fixed. Then for each $\varepsilon>0,k\in\mathbb{N}$ there is $\delta$ such that if $A$ is a  $P$-operator with $\|A\|_{p\to q}\leq c$ and $\mu\in\mathcal{S}_k(A)$ satisfies $d_{\rm LP}(\mu,M_k)\leq\delta$, then there is $\mu_2\in\mathcal{S}'_k(A)$ with $d_{\rm LP}(x,y)\leq\varepsilon$. Furthermore, if for two $P$-operators $A,B$ with $\|A\|_{p\to q},\|B\|_{p\to q}\leq c$ we have $d_H(\mathcal{S}_k(A),\mathcal{S}_k(B))\leq\delta/2$, then $$d_H(\mathcal{S}'_k(A),~\mathcal{S}'_k(B))\leq\varepsilon+2d_H(\mathcal{S}_k(A),\mathcal{S}_k(B)).$$
\end{lemma}

\begin{proof} Assume that $A\in\mathcal{B}(\Omega)$. Let $\mu=\mathcal{D}_A(\{v_i\}_{i=1}^k)$. If $\varepsilon_2>0$ and $\delta$ is small enough, then we have that there are $0-1$-valued functions  $w_1,w_2,\dots,w_k$ in $L^p(\Omega)$ such that $\sum_{i=1}^k w_i=1_{\Omega}$ and $\max_{i\in [k]}\|v_i-w_i\|_p\leq\varepsilon_2$. It is clear that if $\varepsilon_2$ is small enough, then $\mu_2:=\mathcal{D}_A(\{w_i\}_{i=1}^k)$ is in $S'_k(A)$ and $d_{\rm LP}(\mu,\mu_2)\leq\varepsilon$. 

To see the second claim let $\mu\in \mathcal{S}'_k(A)$. Then there is some $\mu_2\in\mathcal{S}_k(B)$ with $d_{\rm LP}(\mu,\mu_2)\leq 2d_H(\mathcal{S}_k(A),\mathcal{S}_k(B))\leq\delta.$ It follows from the first part of the statement that there is $\mu_3\in\mathcal{S}'_k(B)$ with $d_{\rm LP}(\mu_2,\mu_3)\leq\varepsilon$. We get that $d_{\rm LP}(\mu,\mu_3)\leq \varepsilon+2d_H(\mathcal{S}_k(A),\mathcal{S}_k(B)).$
\end{proof}

\begin{lemma}  Let $p\in [1,\infty),q\in (1,\infty], c\in\mathbb{R}^+$ be fixed.  For every $\varepsilon>0,k\in\mathbb{N}$ there is $\delta>0$ and $k'\in\mathbb{N}$ such that if $A\in\mathcal{B}(\Omega_1)$ and $B\in\mathcal{B}(\Omega_2)$ are $P$-operators with $\|A\|_{p\to q},\|B\|_{p\to q}\leq c$ and $d_H(\mathcal{S}'_{k'}(A),\mathcal{S}'_{k'}(B))\leq\delta$, then $d_H(\mathcal{S}_k(A),\mathcal{S}_k(B))\leq\varepsilon$.
\end{lemma}

\begin{proof} Let $\mu\in\mathcal{S}_k(A)$ be given by $\mu=\mathcal{D}_A(\{v_i\}_{i=1}^k)$, where $v_i\in L^\infty_{[-1,1]}(\Omega_1)$ holds for $i\in [k]$. For an arbitrary natural number $m$ and $v\in L^p(\Omega)$ let $[v]_m$ denote the $m^{-1}\mathbb{Z}$-discretization of $v$ obtained by composing $v$ with the function $x\mapsto \lceil xm\rceil/m$. It is clear that $\|v-[v]_m\|_p\leq m^{-1}$. For every $i\in[k]$ the level sets of $[v_i]_m$ partition $\Omega_1$ into at most $2m$ measurable sets. By taking common refinement of the level sets we obtain that there is a partition $\{P_i\}_{i=1}^N$ of $\Omega_1$ into $N\leq (2m)^k$ measurable sets such that each $[v_i]_m$ is measurable in this partition. This means that there exist real numbers $\{u_{i,j}\}_{i\in [k],j\in[N]}$ between $-1$ and $1$ such that for every $i\in[k]$ we have $[v_i]_m=\sum_{j\in [N]}u_{i,j}1_{P_j}$. 

Let $\varepsilon'>0$ be some sufficiently small number. If $\delta$ is small enough, we have that there is a partition $\{Q_i\}_{i=1}^N$ of $\Omega_2$ such that $\mathcal{D}_A(\{1_{P_i}\}_{i=1}^N)$ and $\mathcal{D}_B(\{1_{Q_i}\}_{i=1}^N)$ are at most $\varepsilon'$ far in $d_{\rm LP}$. Let $w_i:=\sum_{j\in [N]}u_{i,j}1_{Q_j}$. It is clear that if $\varepsilon'$ is small enough, then $\kappa:=\mathcal{D}_B(\{w_i\}_{i=1}^k)$ is arbitrarily close to $\mathcal{D}_B(\{v_i\}_{i=1}^k)$. We obtain that if $m$ is big enough, and $\varepsilon'$ is small enough then $\kappa\in\mathcal{S}_k(B)$ is at most $\varepsilon$ far from $\mu$. All the estimates in the proof depend only on $c,\varepsilon,k$ and $p,q$.

\end{proof}

\section{Dense graph limits and graphons}

In this chapter we explain how the so-called dense graph limit theory fits into our general limit theory. Let us consider the probability space $([0,1],\mathcal{L},\lambda)$, where $\lambda$ is the Lebesgue measure on the Lebesgue $\sigma$-algebra $\mathcal{L}$. Special $P$-operators on $L^2([0,1])$ -- called graphons -- play a crucial role in dense graph limits. A {\bf graphon} is a two variable measurable function $W:[0,1]^2\to [0,1]$ with the symmetry property that $W(x,y)=W(y,x)$ holds for every $x,y\in [0,1]$. Graphons act on the Hilbert space $L^2([0,1])$ by 
$$(fW)(x):=\int_y f(y)W(y,x)~d\lambda,$$ where $f\in L^2([0,1])$. It is easy to see that $\|W\|_{2\to 2}\leq 1$ and thus graphons are $P$-operators. It is also clear that graphons are positivity-preserving and self-adjoint operators and hence they are also graphops. Let $\mathcal{W}$ denote the space of graphons. For $U,W\in\mathcal{W}$ we say that $U\sim W$ ($U$ is isomorphic to $W$) if $\delta_\square(U,W)=0$. Let $\tilde{\mathcal{W}}:=\mathcal{W}/\,\sim$ be the set of equivalence classes. The next theorem from \cite{LSz} is a fundamental result in graph limit theory.

\begin{theorem}\label{densecomp}[Lov\'asz--Szegedy, \cite{LSz}] The metric space $(\tilde{\mathcal{W}},\delta_\square)$ is compact.
\end{theorem}

The space $(\tilde{\mathcal{W}},\delta_\square)$ is basically the graph limit space. Every finite graph $G=(V(G),E(G))$ with $V(G)=[n]$ is represented in $\tilde{\mathcal{W}}$ by the function $W_G$ with $W_G(x,y):= 1$ if $(\lceil xn\rceil,\lceil yn\rceil)\in E(G)$ and $W_G(x,y):=0$ otherwise. Graph convergence in dense graph limit theory is equivalent to the convergence of the representing functions $W_G$ in $(\tilde{\mathcal{W}},\delta_\square)$.
Our next theorem shows that this limit theory is embedded into our more general limit framework. %In particular the dense graph limit space $\tilde{\mathcal{W}}$ is embedded into the much larger space $\mathcal{X}_1$ introduced in chapter...

\begin{theorem}\label{dmequiv} The two pseudometrics $\delta_\square$ and $d_M$ are equivalent on $\mathcal{W}$. 
\end{theorem}

\begin{proof} By Lemma \ref{limspdm2} it remains to show that for every $\varepsilon>0$ there is $\delta>0$ such that if $d_M(U,W)\leq\delta$, then $\delta_\square(U,W)\leq\varepsilon$. By contradiction, let us assume that there exist $\varepsilon>0$ and two sequence of graphons $\{U_i\}_{i=1}^\infty$ and $\{W_i\}_{i=1}^\infty$ such that $\delta_\square(U_i,W_i)>\varepsilon$ for every $i\in\mathbb{N}$ and $\lim_{i\to\infty}d_M(U_i,W_i)=0$. By choosing an appropriate subsequence we can assume by Theorem \ref{densecomp} that $\lim_{i\to\infty} U_i=U$ and $\lim_{i\to\infty}W_i=W$ holds where the convergence is in $\delta_\square$. We obtain that $\delta_\square(U,W)\geq\varepsilon$. On the other hand by the triangle inequality and Lemma \ref{limspdm2} we have that 
\begin{align*}d_M(U,W)&\leq d_M(U,U_i)+d_M(U_i,W_i)+d_M(W_i,W)\\&\leq d_M(U_i,W_i)+12\delta_\square(U,U_i)^{1/2}+12\delta_\square(W,W_i)^{1/2}\end{align*} and thus by taking $\lim_{i\to\infty}$ we get that $d_M(U,W)=0$. We have by Proposition \ref{quprop} that $d_{1,H}(\mathcal{Q}_k(U),\mathcal{Q}_k(W))=0$ holds for every $k\in\mathbb{N}$. It is well known in graph limit theory (see e.g. \cite{BCCZ2})  that such quotient equivalence implies $\delta_\square(U,W)=0$.
\end{proof}

The next lemma is rather technical and we skip the proof.

\begin{lemma} For every $\varepsilon>0$ there exists a number $n$ such that if $G$ is a finite graph with $|V(G)|\geq n$, then $d_M(A(G)/|V(G)|, W_G)\leq\varepsilon$.
\end{lemma}

Together with Theorem \ref{dmequiv} it implies the following.

\begin{proposition} If $\{G_i\}_{i=1}^\infty$ is a growing sequence of finite graphs, then the action convergence of $\{A(G_i)/|V(G_i)|\}_{i=1}^\infty$ is equivalent to dense graph convergence. 
\end{proposition}

%\begin{proof} Let $A:=A(G)/V(G)$ and assume that $V(G)=[m]$ for some $m\geq n$. It is easy to see that $A\prec W_G$. Indeed, if $\mu\in\mathcal{S}_k(A)$ is given by functions $\{v_i\}_{i=1}^k$ on $[m]$ then we can simulate this system by the vectors $v_i'(x):=v_i(\lceil xm\rceil)$. It is clear that $\mathcal{D}(\{v_i'\}_{i=1}^k,\{v_i'W_G\}_{i=1}^k)=\mu$.

%It is well know that $A$ can be approximated in the cut norm $\|A\|_{\square}$ by a so-called step matrix $A'$ with boundedly many steps. This is basically the so-called weak regularity lemma (CITE..). The matrix $A'$ has the form that there is partition of $[m]$ into $t$ sets of roughly equal size and $A'_{i,j}$ depends only on the pertition sets containing $i$ and $j$.  By lemma \ref{limspdm} we obtain that $d_M(A,A')$ is small. Let $W':=W_{mA'}$. Similarly we obtain that $d_M(W_{mA'},W_G)$ is small. It remains to show that every measure  $\mu\in\mathcal{S}_k(W_{A'})$ can be approximated by some measure in $\mathcal{S}_k(A)$. 
%\end{proof}

\section{Benjamini--Schramm and local-global limits}

Benjamini--Schramm and local-global limits are used in the study of bounded degree graphs. Let $d$ be a fixed number and let $\mathcal{G}_d$ denote the set of isomorphism classes of graphs with maximal degree at most $d$. Informally speaking, a graph sequence $\{G_i\}_{i=1}^\infty$ in $\mathcal{G}_d$ is Benjamini--Schramm convergent if for every fixed $r$ the probability distribution of isomorphism classes of neighborhoods of radius $r$ converges when $i$ goes to infinity. It is often useful to refine this convergence notion to the local-global setting. In this framework we put "colorings" on the vertex sets of $G_i$ in all possible ways and look at all possible colored neighborhood statistics. It is not to confuse with colored Benjamini--Schramm limits, where we put one coloring on each graph $G_i$. We give the formal definition of local-global limits.

We summarize the notion of local-global convergence based on \cite{HLSz}. A rooted graph is a graph with a distinguished vertex $o$ called {\it root}. The radius of a rooted graph is the maximal distance from the root over all vertices. A $k$-coloring of a graph $G=(V,E)$ is a function $f:V\to [k]$. 
Let $\mathcal{G}_{d,k,r}$ denote the set of isomorphism classes of $k$-colored rooted graphs of maximal degree at most $d$ and radius at most $r$. Note that $\mathcal{G}_{d,k,r}$ is a finite set. We denote by $\mathcal{P}(\mathcal{G}_{d,k,r})$ the set of probability distributions on $\mathcal{G}_{d,k,r}$. We have that $\mathcal{P}(\mathcal{G}_{d,k,r})$ together with the total variation distance $d_{\rm TV}$ is a compact metric space. By abusing the notation we denote by $d_H$ the Hausdorff distance for subsets in $(\mathcal{P}(\mathcal{G}_{d,k,r}),d_{\rm TV})$. 

Let $G=(V,E)\in\mathcal{G}_d$ and let $f:V\to [k]$ be a $k$-coloring. We denote by $\tau_r(G,f)$ the probability distribution on $\mathcal{G}_{d,k,r}$ obtained by putting the root $o$ on a uniformly chosen random vertex of $G$ and then taking the colored neighborhood of $o$ of radius $r$. Let $Z_{k,r}(G)$ denote the set of all possible probability distributions $\tau_r(G,f)$, where $f$ runs through all possible $k$-colorings of $G$. We have that $Z_{k,r}(G)$ is a subset of $\mathcal{P}(\mathcal{G}_{d,k,r})$.

\begin{definition} A graph sequence $\{G_i\}_{i=1}^\infty$ is called local-global convergent if for every $r,k\geq 1$ the sequence $Z_{k,r}(G_i)$ is convergent in the metric $d_H$.
\end{definition}

It was proved in \cite{HLSz} that limits of local-global convergent graph sequences can be described by certain Borel graphs called graphings. We give the formal definition.

\begin{definition} {\bf (Graphing)}  Let $X$ be a Polish topological space and let $\nu$ be a probability measure on the Borel sets in $X$. A graphing is a graph $G$ on $V(G)=X$ with Borel measurable edge set $E(G)\subset X\times X$ in which all degrees are at most $d$ and 
\begin{equation}\label{gmeaspres}
\int_A e(x,B)d\nu(x)=\int_B e(x,A)d\nu(x)
\end{equation}
 for all measurable sets $A,B\subseteq X$, where $e(x,S)$ is the number of edges from $x\in X$ to $S\subseteq X$.
\end{definition}

A $k$-coloring of a graphing $G=(X,E)$ is a measurable function $f:X\to [k]$. The probability measure $\nu$ allows us to talk about random vertices in $G$. The colored neighborhood of a random vertex in $G$ is a graph in $\mathcal{G}_{d,k,r}$. The probability distribution $\tau_r(G,f)$, the  measure  $Z_{k,r}(G)$ and convergence are similarly defined as for finite graphs. 

 Finite graphs are special graphings, where $X$ is a finite set and $\nu$ is the uniform distribution.
It will be important that graphings are bounded operators on $L^2(X,\nu)$. The action is given by $$(vG)(x)=\sum_{(y,x)\in E}v(y)$$ for $v\in L^2(X,\nu)$. We have that $\|G\|_2\leq d$. Note that the integral formula (\ref{gmeaspres}) is equivalent to the fact that $G$ is a self-adjoint operator. Graphings are also positivity-preserving and hence they are examples for graphops. The next theorem is proved in \cite{star}. 

\begin{theorem}\label{starconv} A sequence of graphings $\{G_i\}_{i=1}^\infty$ is local-global convergent if and only if $Z_{k,1}(G_i)$ is convergent in $d_H$ for every fixed $k\geq 1$.
\end{theorem}

Our main theorem here says that, restricted to graphings, $P$-operator convergence is the same as local-global convergence. Consequently, $P$-operator convergence is a generalization of graphing convergence.

\begin{theorem}\label{locglobeq} A sequence of graphings is local-global convergent if and only if it is convergent in the metric $d_M$.
\end{theorem}

\begin{proof} We need some preparation. Let $M_{d,k}$ denote the set of vectors $v=(v_1,v_2,\dots,v_{2k})$ in $\mathbb{N}_0^{[2k]}$ (where $\mathbb{N}_0=\mathbb{N}\cup\{0\}$) such that $\sum_{i=1}^k v_i=1$ and $\sum_{i=k+1}^{2k}v_i\leq d$. There is a natural bijection $\alpha$ between $M_{d,k}$ and $\mathcal{G}_{d,k,1}$ given in the following way. For a vector $v\in M_{d,k}$ let $q(v)$ denote the unique coordinate $i\in [k]$ with $v_i=1$ and let $s(v):=\sum_{i=k+1}^{2k}v_i$.  We denote by $\alpha(v)$ the colored star in which the color of the  root is $q(v)$, the root has $s(v)$ neighbors, and for the neighbors of $o$ the color $i\in[k]$ is used $v_{i+k}$ times. It is clear that the isomorphism type $\alpha(v)$ is determined by this information and each isomorphism type in $\mathcal{G}_{d,k,1}$ is obtained this way. Consequently $\alpha$ is a bijection. We denote by $\hat{\alpha}$ the bijection between the sets of probability measures $\mathcal{P}(M_{d,k})$ and $\mathcal{P}(\mathcal{G}_{d,k,1})$ induced by $\alpha$ using the formula $\hat{\alpha}(\mu)(T):=\mu(\alpha^{-1}(T))$. It is clear that $\hat{\alpha}$ is continuous with respect to the given metrization on $\mathcal{P}(M_{d,k})$ and $\mathcal{P}(\mathcal{G}_{d,k,1})$.

Observe that if $G$ is a graphing (of maximal degree $d$), then $\mathcal{S}'_k(G)=\mathcal{S}_k(G)\cap \mathcal{P}(M_{d,k})$. Now we show that $\hat{\alpha}(\mathcal{S}'_k(G))=Z_{k,1}(G)$. To see this, notice that colorings $f:X\to [k]$ are in a one-to-one correspondence with systems of $0-1$-valued functions $\{v_i\}_{i=1}^k$ with $\sum_{i=1}^k v_i=1_X$, called function partitions. The correspondence is given by $v_i=1_{f^{-1}(i)}$. It is clear that if $\{v_i\}_{i=1}^\infty$ is a function partition, then $$\hat{\alpha}(\mathcal{D}_G(\{v_i\}_{i=1}^k))=\tau_1(G,f).$$ We obtain that $\hat{\alpha}(\mathcal{S}'_k(G))=\hat{\alpha}(\mathcal{S}_k(G)\cap\mathcal{P}(M_{d,k}))=Z_{k,1}(G).$  Now the continuity of $\hat{\alpha}$ and Theorem \ref{partthm} finish the proof.
\end{proof}

%\section{Measures of sparsity for P-operators}

%Graphings and graphons are both P-operators but intuitively they represent two opposite ends of a spectrum of sparsity. On one hand side graphons are compacts operator and can be represented by bounded integral kernel operators. On the other hand graphings are non compact and are represented by a rather singular measure on $\Omega^2$. Intuitively graphons are 2-dimensional objects while graphings are 1-dimensional. This resonates very well with the fact that graphons represent limits of dense graph sequences in which the number of edges grows quadratically while graphings are limits of very sparse graphs sequences with linear edge number growth. Is there a way to interpolate between the two cases? Can we introduce a fractional dimension of P-operators which is somehow related to sparsity. To fit well into our limit theory it is crucial that such a notion is invariant with respect to weak isomorphism of P-operators.

 % Let $A\in\mathcal{B}_2(\Omega,\mathcal{A},\mu)$ be a P-operator. We define the local sparsity exponent $\beta_{\rm loc}(A)$ by
%$$\beta_{\rm loc}(A):=\lim_{\epsilon\to 0}~\inf_{\mu(T)\leq\epsilon}~ \log \|1_TA\|_2/\log \|1_T\|_2$$

\section{Generalizations}

Action convergence is  based on a very general principle. We do not exploit the generality of it in this paper, but as illustration we describe a few useful generalizations.

\noindent{\bf Complex spaces.}~The theory developed in this paper can be generalized to operators acting on complex number valued function spaces. Most of the definitions are the same and the proofs of the theorems require some minor changes. Note also that if a $P$-operator $A$ over $\mathbb{C}$ has the special property that it takes real valued functions to real valued functions, then its $k$-profile over $\mathbb{C}$-valued functions can be reconstructed from its $2k$-profile over $\mathbb{R}$ by decomposing functions according to real and complex part. 

\noindent{\bf Simultaneous convergence.}~We have mentioned in the introduction that it is sometimes useful to introduce simultaneous convergence of pairs $(A,f)$, where $A\in\mathcal{B}(\Omega)$ is a $P$-operator and $f$ is measurable on $\Omega$. Based on the same principle one can further generalize this to a simultaneous convergence notion of several $P$-operators and several functions. An especially interesting case is when matrices and their adjoint matrices are considered simultaneously. In this case $\mathcal{S}_k(A)$ is defined by
$$\mathcal{D}(v_1,v_2,\dots,v_k,v_1A,v_2A,\dots,v_kA,v_1A^*,v_2A^*,\dots,v_kA^*).$$ It is not clear whether this leads to a finer convergence notion of matrices or not.

\noindent{\bf Non-linear operators.}~In the definition of the metric $d_M$ we never use the linearity of the operators. The definition of $k$-profile and distance $d_M(A,B)$ is meaningful for arbitrary functions $A:L^\infty(\Omega_1)\to L^1(\Omega_1)$ and $B:L^\infty(\Omega_2)\to L^1(\Omega_2)$. (Even multivalued functions can be allowed.) Interesting examples for non-linear operators are finite matrices composed pointwise with non-linear functions. For example: $(x,y,z)A:=((x+y)^2, \sin(y+z),z-x)$. Such functions arise in deep learning.

\section{Random matrices}

In this section we investigate the convergence of certain dense random matrices with respect to the metric $d_M$. We consider a sequence of normalized random matrices $(H_n)$ with independent zero-mean $\pm n^{-1/2}$-valued random variables as entries. This is the same as if we choose an element uniformly at random from the the set of all $n\times n$ matrices with $\pm n^{-1/2}$ entries. This set will be denoted by $\mathcal M_n$. Our goal is to prove the following. 

\begin{proposition}\label{prop:wn} For every infinite $S'\subseteq \mathbb N$ there exists  a $P$-operator $A$ and an infinite set  $S\subseteq S'$   such that the sequence $(H_j)_{j\in S}$ converges to $A$ with respect to $d_M$ with probability $1$. 
\end{proposition}

We start with a statement on the concentration of measure. 

\begin{lemma} \label{lem:azuma} For every $n\in \mathbb N$, let $M_n\in \mathcal M_n$ be fixed, and $H_n$ be a uniformly chosen element of $\mathcal M_n$. Then for every $\eta>0$,  we have  
\[\lim_{n\rightarrow\infty}\mathbb P\big(\big|d_M(M_n, H_n)-\mathbb E(d_M(M_n, H_n))\big|>\eta\big)=0.\]
\end{lemma}
\begin{proof} For $1\leq j\leq n$, let $\mathcal F_j$ be the $\sigma$-algebra generated by the first $j$ columns of $H_n$. We apply the well-known concentration inequalities for the martingale $Y_j=\mathbb E\big(d_M(M_n, H_n)\big|\mathcal F_j\big)$. Notice that if matrices $A, B\in \{-1,1\}^{n\times n}$ differ only in a single column, then the distance of $\mathcal D_{A}(v_1, \ldots, v_k)$ and $\mathcal D_{B}(v_1, \ldots, v_k)$ is at most $1/n$ in the L\'evy--Prokhorov metric (for arbitrary $k$ and vectors $v_j$), because the two measures coincide everywhere except on an event of probability $1/n$.  This implies  $d_M(A, B)\leq 1/n$. Hence  $|Y_j-Y_{j-1}|\leq 1/n$ holds for every $j=1, 2, \ldots, n$. Therefore by Azuma's inequality we have that 
\[\mathbb P\big(\big|d_M(M_n, H_n)-\mathbb E(d_M(M_n, H_n))\big|>\eta\big)\leq 2\exp\bigg(-\frac{\eta^2}{2n\cdot (1/n)^2}\bigg)=2\exp(-\eta^2 n/2).\qedhere\]  
\end{proof}

\begin{lemma} \label{lem:mn} There exists a sequence of matrices $(M_j)_{j\in \mathbb N}$ such that the following conditions hold: $(i)$ $M_j\in \mathcal M_j\cap \mathcal X_{2,2,3}$ for every $j\in \mathbb N$; $(ii)$ $d_M(M_j, H_j)\rightarrow 0$ in probability as $j\rightarrow\infty$, where $H_j$ is a uniformly chosen random element of $\mathcal M_j$. 
\end{lemma}
\begin{proof} Given $\varepsilon>0$, first we find a sequence  of matrices around which random matrices are concentrated with error $\varepsilon$. 
The metric space $(\mathcal X_{2,2,3}, d_M)$ is compact by Theorem \ref{thm:compact}, hence it contains a finite $\varepsilon/8$-net. We denote the size of this net by  $F(\varepsilon)$.  Consider balls of radius $\varepsilon/8$ around the elements of this net. Let $\mathcal N_{\varepsilon,n}$ be the set of matrices  satisfying the following property: in one of these balls, it is the closest element of $\mathcal M_n$ to the center (in case of equality, choose one arbitrarily). Then $\mathcal N_{\varepsilon, n}$ is an $\varepsilon/4$-net in $\mathcal M_n\cap \mathcal X_{2,2,3}$, and its size is at most $F(\varepsilon)$, as we have chosen at most one element from each ball. It follows that there exists $M'_{\varepsilon, n}\in \mathcal N_{\varepsilon,n}$ such that
\[\mathbb P\big(d_M(M'_{\varepsilon, n}, H_n)\leq \varepsilon/4\big)\geq \frac{1-\mathbb P(\|H_n\|_2>3)}{F(\varepsilon)}.\]  

Since the operator norm of our random matrix $H_n$ random matrix is concentrated around its expectation $2$ (see e.g.\ \cite{geman}), the probability  $\mathbb P(\|H_n\|_2>3)$ tends to $0$ as $n$ goes to infinity. Therefore for every $\varepsilon>0$, we have 
\[\liminf_{n\rightarrow\infty}\mathbb P\big(d_M(M'_{\varepsilon, n}, H_n)\leq \varepsilon/4)>0.\]
This equation together with Lemma \ref{lem:azuma} for $\eta=\varepsilon/4$ and $(M'_{\varepsilon, n})_{n\in \mathbb N}$ implies that 
\[\mathbb E\big(d_M(M'_{\varepsilon,n}, H_n)\big)\leq \frac{\varepsilon}{2}.\]
By combining this with  Lemma \ref{lem:azuma} for $\eta=\varepsilon/2$, we conclude that 
\[\lim_{n\rightarrow\infty}\mathbb P\big(d_M(M'_{\varepsilon,n}, H_n)>\varepsilon\big)=0.\]
The proof can be completed by a standard diagonalization argument. More precisely, we can choose a function $n_0(\varepsilon)$ such that 
\[\mathbb P\big(d_M(M'_{\varepsilon,n}, H_n)>\varepsilon\big)<\varepsilon \text{ holds for every } n\geq n_0(\varepsilon).\]
Now let $k(n)=\max\{k: n_0(1/k)<n\}$. Then the sequence $M_j=M'_{1/k(j),j}$ satisfies the conditions of the lemma. 
\end{proof}

\begin{proof}[Proof of Proposition \ref{prop:wn}] Let $(M_j)_{j\in \mathbb N}$ be a sequence of matrices satisfying the conditions of Lemma \ref{lem:mn}. By this lemma, we can choose an infinite subset $S\subseteq S'$ such that $(d_M(M_j, H_j))_{j\in S}$ tends to $0$ with probability $1$ as $j\rightarrow\infty$, and $(M_j)_{j\in S}$ converges to a $P$-operator $A$ with respect to $d_M$. To guarantee the second condition, we can use Lemma \ref{lem:sc} and Theorem \ref{exlim}, because $M_j\in \mathcal X_{2,2,3}$ for all $j$. This $S$ will be an appropriate subset of $S'$.
\end{proof}

\section{Examples}\label{chapex}

\subsection{Hypercubes and uniform towers}

The hypercube graph $Q_n=(V(Q_n),E(Q_n))$ is formed by the vertices and edges of the $n$-dimensional hypercube. More precisely, $V(Q_n)=\{0,1\}^n$ and two vertices are connected if and only if the representing vectors have Hamming distance one, i.e. they differ at exactly one coordinate. The graph $Q_n$ is $n$-regular, $|V(Q_n)|=2^n$ and $|E(Q_n)|=2^{n-1}n$. This means that  the sequence $\{Q_n\}_{n=1}^\infty$ is very sparse but not with bounded degrees. 
Note that $Q_n$ is a Cayley graph of the group $Z_2^n$ where $\{0,1\}$ is identified with the cyclic group $Z_2$ of order $2$ and the generators are the basis vectors $e_i, i\in [n]$ with $1$ at the $i$-th coordinate and $0$ elsewhere. 

\noindent{\it Our goal is to show that hypercubes converge to an appropriate Cayley graph of the compact group $Z_2^\infty$ with a carefully chosen topological basis. }

A topological basis is an independent set of vectors in $Z_2^\infty$ that generates a dense set in $Z_2^\infty$. (Note that topological independence is not assumed here.) Quite surprisingly the usual topological basis $\{e_i\}_{i=1}^\infty$ of $Z_2^\infty$ is not useful for constructing the limit of the hypercubes. The main obstacle is that $\{e_i\}_{i=1}^\infty$ is a countable set but there is no natural uniform distribution on an infinite countable set. Instead we need to find a nice enough topological basis with uncountable many elements and a natural uniform distribution on this basis.

Since $Q_n$ is regular, we have that adjacency operator convergence is equivalent to random walk convergence, so we do not have to choose one of them. The right scaling of the sequence is  $A_n:=A(Q_n)/n$, where $A(Q_n)$ is the adjacency matrix of $Q_n$. The operator $A_n$ is a Markov graphop, and if $\{A_n\}_{n=1}^\infty$ is convergent, then the limit is also a Markov graphop (recall Theorem \ref{markcomp}). As we stated above, the purpose of this part of the paper is to show that they indeed converge and to determine the limit object. Some details will be left to the reader regarding the general convegence. We will work with the subsequence $\{Q_{2^n}\}_{i=1}^\infty$ that has especially nice properties based on certain  uniform mappings between $Q_{2^{n+1}}$ and $Q_{2^n}$. The general convergence can be obtained from approximate versions of these uniform maps. 

%\counterwithout{figure}{section}
\begin{figure}[htbp] 

  \centering
       \includegraphics[scale=0.45]{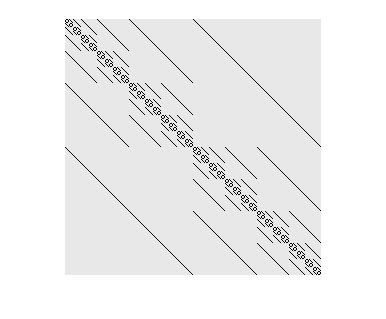}  \includegraphics[scale=0.45]{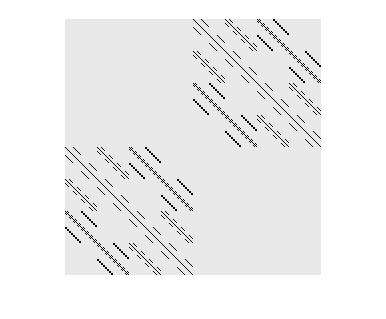}
  \caption{Two representations of the adjacency matrix of the hypercube in dimension 8}\label{figcube}
 \end{figure}

On Figure \ref{figcube} we show the adjacency matrix of the $8$ dimensional hypercube $Q_8$ using two different orderings of the vertices. Light gray points represent zeros and black points represent ones. The first ordering is based on the binary forms of the numbers $0,1,2,\dots,255$, which is a rather natural way to order $\{0,1\}^8$. On the second figure we compose this ordering with a carefully chosen automorphism of the group $Z_2^8$. Quite surprisingly it turns out that the second figure provides a more useful representation when going to the limit. There is a qualitative difference between the two types of representations of $Q_8$. Intuitively, the first pictures would converge to some "infinite picture", where each vertical (and horizontal) line has countable intersection with the black points. On the other hand the second figure fits into a sequence such that, after going to the limit, vertical (and horizontal) lines have uncountable intersections with the black points. We will see later that this helps in putting a uniform distribution on the limiting picture.

We will need the following definition. 

\begin{definition}{\bf (Uniform map and uniform tower)} Let $G_1,G_2$ be graphs. A map $f:V(G_2)\to V(G_1)$ is $(a,b)$-uniform if 
\begin{enumerate}
\item $f$ is a graph homomorphism, i.e. $(f(v),f(w))\in E(G_1)$ holds for every $(v,w)\in E(G_2)$.  
\item $|f^{-1}(v)|=a$ holds for every $v\in V(G_1)$.
\item If $v\in V(G_2)$ and $w$ is any neighbor of $f(v)$, then exactly $b$ neighbors of $v$ are mapped to $w$.
\end{enumerate}
A uniform tower is a sequence $\{G_i,f_i\}_{i=1}^\infty$ of finite graphs $G_i$ and maps $f_i:V(G_{i+1})\to V(G_i)$ such that $f_i$ is $(a_i,b_i)$-uniform for $i\in\mathbb{N}$.
\end{definition}

\begin{lemma}\label{hyplem1} Let $G_1,G_2$ be finite graphs and let $f:V(G_2)\to V(G_1)$ be an $(a,b)$-uniform map for some $a,b\in\mathbb{N}$. Then $A(G_1)\prec A(G_2)/b$.
\end{lemma}

\begin{proof} The second property of uniformity implies that if $v_1,v_2,\dots,v_k\in \mathbb{R}^{V(G_1)}$ for some $k$, then $$\mathcal{D}(v_1,v_2,\dots,v_k)=\mathcal{D}(v_1\circ f,v_2\circ f,\dots,v_k\circ f).$$
Furthermore, the first and third property imply that if $v\in\mathbb{R}^{V(G_1)}$, then $vA(G_1)\circ f=(v\circ f)A(G_2)/b$. We obtain that 
$$\mathcal{D}_{A(G_1)}(\{v_i\}_{i=1}^k)=\mathcal{D}_{A(G_2)/b}(\{v_i\circ f\}_{i=1}^k),$$ and hence $\mathcal{S}_k(A(G_1))\subseteq\mathcal{S}_k(A(G_2)/b)$ holds for every $k$.
\end{proof}

Recall that if $\{X_i\}_{i=1}^\infty$ is a sequence of finite sets with maps $f_i:X_{i+1}\to X_i$, then the inverse limit $X$ is the set of elements in $(x_1,x_2,\dots)\in \prod_{i=1}^\infty X_i$ such that $f_i(x_{i+1})=x_i$ holds for every $i$. Since $X$ is a closed subset of the compact space $\prod_{i=1}^\infty X_i$, we have that $X$ is compact with respect to the subspace topology. The map $\pi_i: X\to X_i$ defined by $\pi_i(x_1,x_2,\dots):=x_i$ is a continuous map. If each $f_i$ has the property that $|f_i^{-1}(v)|=|f_i^{-1}(w)|$ holds for every $v,w\in X_{i+1}$, then there is a unique Borel probability measure $\mu$ on $X$ such that for every $i$ the push-forward measure of $\mu$ under $\pi_i$ is uniform on $X_i$. We call $\mu$ the uniform measure on $X$.

\begin{definition} Let $\{G_i,f_i\}_{i=1}^\infty$ be a uniform tower such that $G_i$ is $d_i$-regular for $i\in\mathbb{N}$. Let $V$ be the inverse limit of $\{V(G_i),f_i\}_{i=1}^\infty$. For every $x\in V$ let $N(x)\subseteq V$ denote the inverse limit of the set of neighbors of $\pi_i(x)$ and let $\nu_x$ denote the uniform measure on $N(x)$. Let $A$ be the $P$-operator in $\mathcal{B}_{2,2}(V,\mu)$ defined by $(fA)(x)=\int_{V} f~d\nu_x$. We say that $A$ is the inverse limit of the tower $\{G_i,f_i\}_{i=1}^\infty$.
\end{definition}

\begin{theorem}\label{hypthm}{\bf (Convergence of uniform towers)} Let $\{G_i,f_i\}_{i=1}^\infty$ be a uniform tower such that $G_i$ is $d_i$-regular for $i\in\mathbb{N}$. Then $\{A(G_i)/d_i\}$ is a convergent sequence of $P$-operators and the limit object is the inverse limit of $\{G_i,f_i\}_{i=1}^\infty$. 
\end{theorem}

\begin{proof} Observe that $\|A(G_i)/d_i\|_2=1$ and if $f_i$ is $(a_i,b_i)$ uniform, then $d_{i+1}=d_ib_i$. We have by Lemma \ref{hyplem1} that $A(G_i)/d_i\prec A(G_{i+1})/d_{i+1}$ holds for every $i\in\mathbb{N}$. It follows by compactness that $S_k(A(G_i)/d_i)$ converges to $\cup_{i=1}^\infty S_k(A(G_i)/d_i)$ in $d_H$ as $i$ goes to infinity. Let $A\in\mathcal{B}_{2,2}(V,\mu)$ be the inverse limit of the tower $\{G_i,f_i\}_{i=1}^\infty$. By approximating measurable functions in $L^2(V,\mu)$ by functions of the form $v\circ\pi_i$ we obtain that $\mathcal{S}_k^*(A)$ is the closure of $\cup_{i=1}^\infty S_k(A(G_i)/d_i)$. 
\end{proof}

\noindent{\bf Construction of the limiting hypercube:}~ We finally arrived to the construction which allows us to determine the limit of the sequence $\{Q_{2^n}\}_{n=1}^\infty$. The main observation is that there are $(2^{2^n},2)$-uniform maps $f_n:V(Q_{2^{n+1}})\to V(Q_{2^n})$. Let $T_n$ denote the vertex set of the rooted binary tree of depth $n$. If $n=\infty$, then $T_\infty$ denotes the infinite rooted binary tree. We have that $T_n$ has $2^n$ leaves. If $v$ is not a leaf, then we denote by $\alpha_1(v)$ and $\alpha_2(v)$ the two children of $v$. Recall that $Z_2$ denotes the group with two elements. For $n\in\mathbb{N}\cup\{\infty\}$ let $G_n$ be the set of functions $f:T_n\to Z_2$ such that $f(v)=f(\alpha_1(v))+f(\alpha_2(v))$. It is clear that $G_n$ is an elementary abelian $2$-group of order $2^{2^n}$ with respect to pointwise addition. It follows that $G_n$ is a vector space of dimension $2^n$ over the field with $2$ elements. The group $G_\infty$ is the inverse limit of the groups $G_n$ and it is a compact abelian group with Haar measure $\mu$. 
Let $B$ denote the boundary of $T_\infty$. It is well known that $B$ is the Cantor set and every element $b\in B$ is uniquely characterized by an infinite path started at the root of $T_\infty$. By abusing the notation let us identify $b$ with this infinite path. Let $g_b$ denote the element in $Z_2^{T_\infty}$ that takes $1$ at the vertices of the path $b$ and $0$ otherwise. It is clear that $g_b\in G_\infty$ holds for every $b\in B$.  Let $Q:=\{g_b:b\in B\}$ and let $\nu$ be the probability measure on $Q$ obtained by first choosing $b$ uniformly in the Cantor set $B$ and then taking $g_b$. For $n\in\mathbb{N}\cup\{\infty\}$ and $m\leq n$ we denote by $\pi_{n,m}:G_n\to G_m$ the group homomorphism obtained by restricting a $Z_2$-labeling of $T_n$ to the subtree $T_m$. It is easy to see that $\pi_{\infty,n}(Q)$ is a basis in the vector space $G_n$ and thus we can represent $Q_{2^n}$ as the Cayley graph of $G_n$ with generators $\pi_{\infty,n}(Q)$. It is easy to see that the maps $$\pi_{n+1,n}:V(Q_{2^{n+1}})\to V(Q_{2^n})$$ are $(2^{2^n},2)$ uniform.

\noindent {\it It follows from Theorem $\ref{hypthm}$ that the limit object of the sequence $\{Q_{2^n}\}_{n=1}^\infty$ is basically the Cayley graph of the compact group $G_\infty\simeq Z_2^\infty$ with generators $Q$ and with uniform measure on the edges. More precisely, let $A$ denote the $P$-operator in $\mathcal{B}_{2,2}(G_\infty,\mu)$ defined by 
$$(fA)(x):=\int_{z\in Q} f(x+z)~d\nu.$$ Then the $P$-operator $A$ is the limit of the graph sequence $\{Q_{2^n}\}$.}

\medskip

\subsection{Product graphs}

The product of two graphs $G_1$ and $G_2$ is the graph on $V(G_1)\times V(G_2)$ such that $((i,j),(k,l))\in E(G_1\times G_2)$ if and only if $(i,k)\in E(G_1)$ and $(j,l)\in E(G_2)$. Graph sequences formed by the powers of a given graph are good test graphs for limit theories. We have that $2|E(G_1\times G_2)|=4|E(G_1)||E(G_2)|$ and thus $E(G^i)=2^{i-1}|E(G_i)|^i$. It follows that $$\beta:=\lim_{i\to\infty} \log |E(G^i)|/\log|V(G^i)|=\log(2|E(G)|)/\log(|V(G)|).$$ The number $0\leq\beta\leq 2$ expresses the exponent of the growth rate of the number of edges in terms of the number of vertices in $\{G^i\}_{i=1}^\infty$. One can view $G^i$ as a fractal like graph (see Figure \ref{figprod}). 

\begin{figure}[htbp]
  \centering
   
   \includegraphics[scale=0.45]{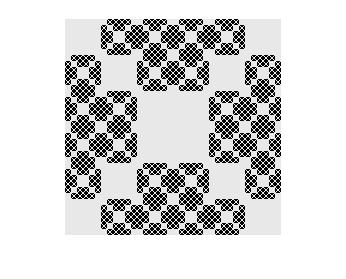}      \includegraphics[scale=0.45]{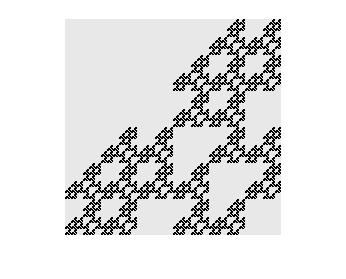}   
  \caption{The 3rd power of two different graphs on six vertices}
  \label{figprod}
\end{figure}

When $G$ is $d$-regular, we can use Theorem \ref{hypthm} to compute the limit object of $\{G^i\}_{i=1}^\infty$. The main observation is that the map $\pi_i:V(G^{i+1})\to V(G^i)$ given by the projection to the first $i$ coordinates is uniform and thus $\{G_i,\pi_i\}_{i=1}^\infty$ is a uniform tower. It is easy to see that the inverse limit is simply given as the infinite power $G^\infty$ with the uniform distribution on the vertices and on the edges. According to Theorem \ref{hypthm} this inverse limit is basically the limit of the normalized $P$-operator sequence $\{A(G^i)/d^i\}_{i=1}^\infty$. The corresponding graphop $A\in\mathcal{B}(V(G)^\infty)$ is given by
$$(vA)(x)=\mathbb{E}_{(x,y)\in E(G^\infty)} v(y),$$ where the expected value is calculated according to the product measure on the neighbors of $x$. More precisely, if $x=(x_1,x_2,\dots)\in V(G^\infty)$ is fixed, then the set of neighbors of $x$ is the infinite product $N(x_1)\times N(x_2)\times\dots$. We define $\nu_x$ as the product of the uniform measures on $N(x_1)\times N(x_2)\times\dots$ and the expected value is according to the measure $\nu_x$. 

If $G$ is not regular, then the degrees in $G^i$ are very unevenly distributed.  In this case we can use random walk convergence to get a non-trivial and natural limit object, but we skip the details.

\medskip

\subsection{Star graphs}

For every $n\geq 1$, let $G_n$ be the star graph on vertex set $\Omega_n=\{0, 1, \ldots, n-1\}$, namely, in which vertex $0$ is connected to every other vertex with a single edge. Since the operator norm of its adjacency matrix is $\sqrt n$, we should normalize by $\sqrt n$ to get a sequence of matrices with bounded operator norm.  In this case the limit will be constant $0$, which does not reflect the structure of the graphs. Therefore, instead of the adjacency operator  convergence notion, we are interested in the random walk convergence of this sequence, as it was defined in Section \ref{general}. 

Let $\nu_n$ be the stationary measure of the random walk on $G_n$. This puts weight $1/2$ to vertex $0$, and $1/(2n-2)$ everywhere else. Then the Markov operator $M_n:=M(G_n)$, which acts on $L^2(\Omega_n)$, is given as follows:
\[(vM_n)(0)=\frac{1}{n-1}\sum_{i=1}^{n-1} v(i); \quad (vM_n)(j)=v(0) \quad (j=1, \ldots, n-1),\]
where $v\in L^2(\Omega_n)$. Hence the $k$-profile of $M_n$ consists of the following probability measures on $\mathbb R^2$, where $v_1, \ldots, v_k\in L^\infty_{[-1,1]}(\Omega_n)$ are chosen arbitrarily: the measure puts weight $1/2$ to 
\[\bigg(v_1(0), \ldots, v_k(0),\, \frac{1}{n-1}\sum_{i=1}^{n-1} v_1(i), \ldots, \frac{1}{n-1}\sum_{i=1}^{n-1} v_k(i)\bigg),\]
and it puts weight $1/(2n-2)$ to 
\[\big(v_1(j), \ldots, v_k(j), v_1(0), \ldots, v_k(0)\big)\]
for each $j=1, 2, \ldots, n-1$. 

Now let $\Omega=[0,1/2]$ with the following probability measure $\nu$: it is the Lebesgue measure on $(0,1/2]$ together with an atom of weight $1/2$ at $0$. We define a $P$-operator $M$ on $L^2(\Omega)$ by 
\[(fM)(0)=2\int_{0}^{1/2} f(y) dy; \qquad (fM)(x)=f(0) \quad (0<x\leq 1/2),\]
where $f\in L^2(\Omega)$. 
Then the $k$-profile of the $P$-operator $M$ is the set of the following probability measures for $f_1, \ldots, f_k\in L^2(\Omega)$: the measure puts weight $1/2$ to 
\[\bigg(f_1(0), \ldots, f_k(0),\, 2\int_0^{1/2} f_1(y)dy, \ldots, 2\int_0^{1/2} f_k(y)dy\bigg),\]
and puts a uniform distribution (with total weight $1/2$) to 
\[\big(f_1(x), \ldots, f_k(x), f_1(0), \ldots, f_k(0)\big).\]

By comparing the profiles of $M_n$ and $M$, for every $k\geq 1$, we have that $\mathcal S_k(M_n)\subset \mathcal S_k(M)$. On the other hand, we show that every element of $\mathcal S_k(M)$ can be approximated weakly by a sequence whose $n$th term is chosen from  $\mathcal S_k(M_n)$. For every $m\geq 1$ we can choose continuous functions $f_1^{(m)}, \ldots, f_k^{(m)}\in L^2(\Omega)$ such that the $L^2$-distance of $f_s$ and $f_s^{(m)}$ is at most $1/m$ for every $s\in [k]$. Furthermore, if $m_n$  is large enough, then by choosing $v_s(j)=f_s(j/m_n)$, we can find an element of $\mathcal S_k(M_{m_n})$ whose L\'evy--Prokhorov distance from  the probability measure corresponding to $f_1, \ldots, f_k$ in $\mathcal S_k(M)$ is arbitrarily small. We conclude that the Hausdorff distance of $\mathcal S_k(M_n)$ and $\mathcal S_k(M)$ tends to $0$ as $n\rightarrow\infty$ for each $k\geq 1$, and $M$ is the limit of the sequence of star graphs with respect to random walk convergence. 

\medskip

\subsection{Subdivisions of complete graphs}

Our second example is the $2$-subdivision of the complete graph on $n$ vertices. More precisely, for $n\geq 1$, let 
\[\Omega_n=[n]\cup \{ w_{ij}, 1\leq i<j\leq n\}.\]
When $j<i$, we will use $w_{ji}=w_{ij}$. As for the edges, for every $1\leq i<j\leq n$, vertex $w_{ij}$ is connected to $i$ and $j$. This graph has $n+n(n-1)/2$ vertices and $n(n-1)$ edges.

We denote by $M_n$ the Markov operator of this graph. For every $v:\Omega_n\rightarrow \mathbb R$ and $1\leq i<j\leq n$, we have 
\[(vM_n)(i)=\frac{1}{n-1}\sum_{j\neq i} v(w_{ij}); \qquad (vM)(w_{ij})=\frac 12\big(v(i)+v(j)\big).\]
The stationary measure puts weight $1/(2n)$ to vertices from $[n]$, and weight $1/(n^2-n)$ to the other vertices. Hence the $k$-profile of $M_k$ is given by the set of probability measures putting weight $1/(2n)$ to 
\[\bigg(v_1(i), \ldots, v_k(i),\, \frac{1}{n-1}\sum_{j\neq i} v_1(w_{ij}),\ldots, \frac{1}{n-1}\sum_{j\neq i} v_k(w_{ij})\bigg),\]
and weight $1/(n^2-n)$ to 
\[\bigg(v_1(w_{ij}), \ldots, v_k(w_{ij}), \, \frac 12\big(v_1(i)+v_1(j)\big), \ldots, \frac 12\big(v_k(i)+v_k(j)\big)\bigg),\]
where $v_1, \ldots, v_k: \Omega_n\rightarrow\mathbb R$ are arbitrary functions.

Let $[0,1]^2_\sim$ denote the set of unordered pairs $\{(x,y):x,y\in [0,1]\}$. In other words $[0,1]^2_\sim$ is the set $[0,1]^2$ factored by the equivalence $(x,y)\sim(y,x)$.
We represent the limit as a $P$-operator on $\Omega=[0,1]\cup [0,1]^2_\sim$. In this case a function $f\in L^2(\Omega)$ can be given by a pair $(f_1, f_2)$, where $f_1\in L^2([0,1])$ and $f_2\in L^2([0,1]^2_\sim)$. 
Then we define $M$ as follows: 
\[(fM)(x)=\int_0^1 f_2(x,u)\, du; \qquad (fM)(y,z)=\frac{f_1(y)+f_1(z)}{2},\]
where $x,y,z$ are all from the interval $[0,1]$. 
The $k$-profile of $M$ consists of probability measures which are the distributions of the following random variables for some functions $f^{(1)}, \ldots, f^{(k)}\in L^\infty_{[-1,1]}(\Omega)$. With probability $1/2$, we choose $x$ uniformly at random from the interval $[0,1]$ and take
\[\bigg(f^{(1)}_1(x), \ldots, f^{(k)}_1(x),\, \int_0^1 f^{(1)}_2(x,u)\, du,\ldots, \int_0^1 f^{(k)}_2(x,u)\, du\bigg).\]
Otherwise we choose $(y,z)\in[0,1]^2$ uniformly at random, and take 
\[\bigg(f^{(1)}_2(y,z),\ldots, f^{(k)}_2(y,z), \, \frac 12\big(f^{(1)}_1(y)+f^{(1)}_1(z)\big), \ldots, \frac 12\big(f^{(k)}_1(y)+f^{(k)}_1(z)\big)\bigg).\]
Similarly to the previous case, by approximating $L^2$ functions with continuous ones, it can be proved that $M$ is indeed the limit of $M_n$, and hence the sequence of $2$-subdivisions of complete graphs converges to this $P$-operator according to random walk convergence.

\medskip

\subsection{Incidence graphs of finite projective planes}

Let $q$ be a prime power and let  $\mathbb{P}(q)$ denote the projective plane over the finite field with $q$ elements. The plane $\mathbb{P}(q)$ has $q^2+q+1$ lines and $q^2+q+1$ points. We denote by $G_q$ the bipartite graph whose vertices are the lines and the points in $\mathbb{P}_q$ and the edges in $G_q$ are incidences in $\mathbb{P}(q)$. This means that a line $l$ is connected with a point $p$ if $l$ contains $p$. We have that $G_q$ is $(q+1)$-regular, $|V(G_q)|=2(q^2+q+1)$ and $|E(G_q)|=(q^2+q+1)(q+1)$.  It follows that the sequence $G_q$ is an intermediate density sequence. The number of edges is roughly the $3/2$-th power of the number of vertices.

\begin{figure}[htbp]
  \centering
   
   \includegraphics[scale=0.3]{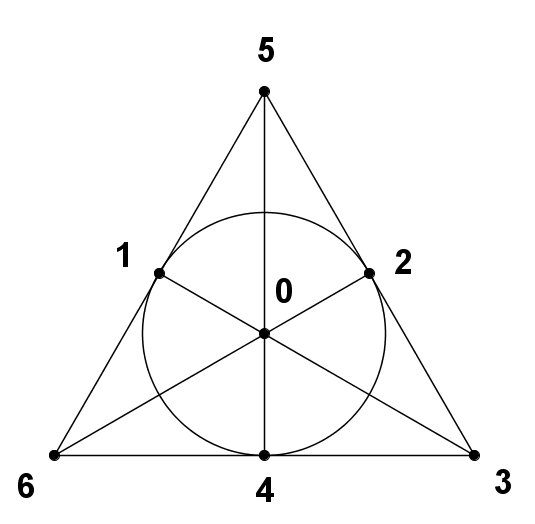}      \includegraphics[scale=0.60]{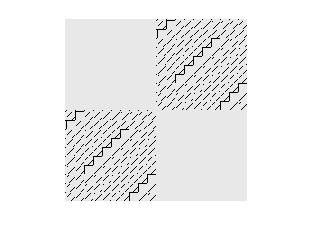}    
  \caption{The Fano plane $\mathbb{P}(2)$ and the matrix of the incidence graph of the projective plane over the field $\mathbb{F}_9$}\label{figfano}
\end{figure}

{\it We show that the matrices $A(G_q)/(q+1)$ form a convergent sequence whose limit is represented by the generalized graphon $W:[0,1]^2\to [0,1]$ defined by $W(x,y)=2$ if $(x,y)\in [1/2,1]\times [0,1/2]\cup[0,1/2]\times [1,1/2]$ and $W(x,y)=0$ elsewhere.}

The proof is based on the fact that the eigenvalues of $G_q$ are known to be $q+1,-q-1,\sqrt{q},-\sqrt{q}$ with multiplicity $1,1,q^2+q,q^2+q$. The two eigenvalues $q+1$ and $-q-1$ belong to the constant $1$ vector $v_1$ and the vector $v_2$ which takes $1$ at points and $-1$ at lines. Let $B_q:=(v_1^*v_1-v_2^*v_2)/2(q^2+q+1)$. We have that $\|A(G_q)/(q+1)-B_q\|_{2\to 2}=q^{-1/2}$. It follows from Lemma \ref{limspdm} that
$$d_M(A(G_q)/(q+1),B_q)\leq 3\|A(G_q)/(q+1)-B_q\|_{2\to 2}^{1/2}=3q^{-1/4}$$ and hence the limit $A(G_q)/(q+1)$ is the same as the limit of $B_q$ as the prime power $q$ goes to infinity. On the other hand $B_q$ is twice the normalized adjacency matrix of the complete bipartite graph with equal color classes on $2(q^2+q+1)$ points. This proves the claim.

The next question illustrates that this does not end the limiting investigation of $G_q$. We can look at it at a finer scale by removing the two dominant eigenvectors and normalizing it with a different constant.

\begin{question} Let $B'_q:=(A(G_q)/(q+1)-B_q)q^{1/2}$.  We have that $\|B'_q\|_{2\to 2}=1$. Does the sequence of $P$-operators $B'_q$ converge as the prime power $q$ goes to infinity? If yes, what is the limit object?
\end{question}

Note that by compactness we know that $B'_q$ has convergent subsequences.

\section{Appendix (technical lemmas)}

\begin{lemma}\label{coupdist} Let $X,Y$ be two jointly distributed $\mathbb{R}^k$-valued random variables. Then $d_{\rm LP}(\mathcal{D}(X),\mathcal{D}(Y))\leq \tau(X-Y)^{1/2}k^{3/4}.$
\end{lemma}

\begin{proof} First we claim that for $a>0$ we have
\begin{equation}\label{eqcoup1}
\mathbb{P}(|X-Y|\geq a)\leq \tau(X-Y)k^{3/2}/a.
\end{equation}
Let $\pi_i:\mathbb{R}^k\to\mathbb{R}$ denote the $i$-th coordinate function for $1\leq i\leq k$. If the square of every coordinate of $X-Y$ is at most $a^2/k$, then $|X-Y|\leq a$. Hence  by the union bound we have that 
$$\mathbb{P}(|X-Y|\geq a)\leq\sum_{i=1}^k\mathbb{P}(|\pi_i(X-Y)|\geq a/k^{1/2})\leq \sum_{i=1}^k \mathbb{E}(|\pi_i(X-Y)|)k^{1/2}/a.$$ Since $\tau(X-Y)$ is the maximum of $\mathbb{E}(|\pi_i(X-Y)|)$ over $1\leq i\leq k$, inequality (\ref{eqcoup1}) follows.

 Let $U$ be a Borel set in $\mathbb{R}^k$. From (\ref{eqcoup1}) we have for every $\varepsilon>0$ that
\begin{equation}\label{eqcoup2}
\mathbb{P}(Y\in U^\varepsilon)\geq \mathbb{P}(X\in U)-c~~~{\rm and}~~~\mathbb{P}(X\in U^\varepsilon)\geq \mathbb{P}(Y\in U)-c,
\end{equation}
 where $c=\tau(X-Y)k^{3/2}\varepsilon^{-1}$. In particular, if $\varepsilon=\tau(X-Y)^{1/2}k^{3/4}$, then $c=\varepsilon$ and (\ref{eqcoup2}) witnesses that $d_{\rm LP}(\mathcal{D}(X),\mathcal{D}(Y))\leq \tau(X-Y)^{1/2}k^{3/4}$.

\end{proof}

\begin{lemma}\label{coupdist2} Let $v_1,v_2,\dots,v_k$ and $w_1,w_2,\dots,w_k$ be in $L^1(\Omega)$ for some probability space $\Omega$. Let $m:=\max_{i\in [k]} \|v_i-w_i\|_1$. Then
$$d_{\rm LP}(\mathcal{D}(v_1,v_2,\dots,v_k),\mathcal{D}(w_1,w_2,\dots,w_k))\leq m^{1/2}k^{3/4}.$$
\end{lemma} 

\begin{proof} We apply Lemma \ref{coupdist} to the jointly distributed random variables $X(\omega):=(v_1(\omega),v_2(\omega),\dots,v_k(\omega))$ and $Y(\omega):=(w_1(\omega),w_2(\omega),\dots,w_k(\omega))$ defined for $\omega\in\Omega$. Since $\tau(X-Y)=m$, Lemma \ref{coupdist} finishes the proof.
\end{proof}

For a real number $z\in\mathbb{R}^+$ let $q_z:\mathbb{R}\to\mathbb{R}$ denote the function such that $f_z(x)=0$ for $|x|\geq 2z$, $f_z(x)=|x+z|-z$ for $x\in [-2z,0]$ and $f_z(x)=-|x-z|+z$ for $x\in [0,2z]$.

\begin{lemma}\label{closedlem1} Let $q\in (1,\infty)$ and let $X$ be a real-valued random variable with $\mathbb{E}(|X|^q)=c<\infty$. Then for $z\in\mathbb{R}^+$ we have that $\mathbb{E} |f_z(X)-X|\leq cz^{1-q}$.
\end{lemma}

\begin{proof} Let $p=q/(q-1)$. We have that $f_z(x)-x$ is $0$ for  $x\in[-z,z]$ and $|f_z(x)-x|\leq |x|$ for $x\in \mathbb{R}\setminus [-z,z]$. It follows from H\"older's inequality that 
\begin{align*}\mathbb{E} |f_z(X)-X|&\leq\mathbb{E}(|X|1_{\mathbb{R}\setminus[-z,z]}(X))\leq \mathbb{E}(|X|^q)^{1/q}\mathbb{E}(1_{\mathbb{R}\setminus[-z,z]})^{1/p}\\&=\mathbb{E}(|X|^q)^{1/q}\mathbb{P}(|X|\geq z)^{1/p}\leq c^{1/q}\mathbb{P}(|X|\geq z)^{1/p}.\end{align*}
By Markov's inequality we have that 
$$\mathbb{P}(|X|\geq z)=\mathbb{P}(|X|^q\geq z^q)\leq\mathbb{E}(|X|^q)/z^q\leq c/z^q.$$ This completes the proof.
 
\end{proof}

\begin{lemma}\label{closedlem2} Let $q\in (1,\infty)$.  Let $\{(X_i,Y_i)\}_{i=1}^\infty$ be a sequence of pairs of jointly distributed real valued random variables such that $X_i\in [-1,1]$ and $\mathbb{E}(|Y_i|^q)\leq c<\infty$ for some $c\in\mathbb{R}^+$. Assume that the distributions of $(X_i,Y_i)$ weakly converge to some probability distribution $(X,Y)$ as $i$ goes to infinity. Then $\mathbb{E}(|Y|^q)\leq c$ and $$\lim_{i\to\infty} \mathbb{E}(X_iY_i)=\mathbb{E}(XY).$$ 
\end{lemma}

\begin{proof}The statement $\mathbb{E}(|Y|^q)\leq c$ follows from the compactness of $\{\mu:\mu\in\mathcal{P}(\mathbb{R}),\int_x |x|^q d\mu\leq c\}$ in the weak topology. 
Since $f_z$ is continuous with finite support, we have that $\lim_{i\to\infty}\mathbb{E}(f_1(X_i)f_z(Y_i))=\mathbb{E}(f_1(X)f_z(Y))$ holds for every $z\in\mathbb{R}^+$. On the other hand we have that
$$|\mathbb{E}(f_1(X_i)f_z(Y_i)-X_iY_i)|=|\mathbb{E}(X_i(f_z(Y_i)-Y_i))|\leq \mathbb{E}|f_z(Y_i)-Y_i|\leq cz^{(1-q)}$$ 
by Lemma \ref{closedlem1}, and similarly
$$|\mathbb{E}(f_1(X)f_z(Y)-XY)|\leq cz^{(1-q)}.$$
It follows that
$$\Bigl|\lim_{i\to\infty} \mathbb{E}(X_iY_i)-\mathbb{E}(XY)\Bigr|\leq 2cz^{(1-q)}$$ and hence as $z$ goes to infinity we obtain the statement of the lemma.
\end{proof}

\begin{lemma}\label{applem1} Let $\mu$ be a probability measure on $[-c,c]$ for some $c\in\mathbb{R}^+$. Let $p\in [1,\infty)$. Then $\int_{\mathbb{R}} |x|^p~d\mu\leq (2d_{\rm LP}(\mu,\delta_{0}))^p+2d_{\rm LP}(\mu,\delta_{0})c^p$
\end{lemma}

\begin{proof} Let $d:=d_{\rm LP}(\mu,\delta_{0})$. We have that $1=\delta_0(\{0\})\leq \mu([-2d,2d])+2d$ and so $\mu([-2d,2d])\geq 1-2d$. It follows that
$$\int_{\mathbb{R}} |x|^p~d\mu=\int_{[-2d,2d]}|x|^p~d\mu+\int_{\mathbb{R}\setminus [-2d,2d]}|x|^p~d\mu\leq (2d)^p+2dc^p.$$
\end{proof}

\begin{lemma}\label{applem2} Let $p\in [1,\infty)$ and let $A\in\mathcal{B}(\Omega)$ be a $P$-operator. Let  $v_i$ and $w_i$ be elements in $L^\infty(\Omega)$ with values in $[-1,1]$ for $i\in [k]$. Then we have
$$d_{\rm LP}(\mathcal{D}_A(\{v_i\}_{i=1}^k),\mathcal{D}_A(\{w_i\}_{i=1}^k))\leq m^{1/2}((2d)^p+2^{p+1}d)^{1/(2p)}(2k)^{3/4},$$ where $m=\max\{1,\|A\|_{p\to 1}\}$ and $d=\max_{i\in [k]}\{d_{\rm LP}(\mathcal{D}(v_i-w_i),\delta_0)\}$.
\end{lemma}

\begin{proof} We have by Lemma \ref{applem1} that $\|v_i-w_i\|_p\leq ((2d)^p+2^{p+1}d)^{1/p}$ for every $i\in[k]$. It follows that $\|v_iA-w_iA\|_p\leq  \|A\|_{p\to 1} ((2d)^p+2^{p+1}d)^{1/p}$ holds for every $i\in[k]$. Then Lemma \ref{coupdist2} finishes the proof.
\end{proof}

\subsection*{Acknowledgement.} The research leading to these results has received funding from the European Research Council under the European Union's Seventh Framework Programme (FP7/2007-2013) / ERC grant agreement n$^{\circ}$617747. The research was partially supported by the MTA R\'enyi Institute Lend\"ulet Limits of Structures Research Group. The first author was supported by the "Bolyai \"Oszt\"ond\'ij" of the Hungarian Academy of Sciences.

\textsc{\'Agnes Backhausz.} ELTE E\"otv\"os Lor\'and University, Budapest, Hungary, Faculty of Science, Department of Probability and Statistics and MTA Alfr\'ed R\'enyi Institute of Mathematics. P\'azm\'any P\'eter s\'et\'any 1/c, Budapest, Hungary, H-1117. \texttt{agnes@math.elte.hu}

\textsc{Bal\'azs Szegedy.} MTA Alfr\'ed R\'enyi Institute of Mathematics, Re\'altanoda utca 13--15., Budapest, Hungary, H-1053.  \texttt{szegedy.balazs@renyi.mta.hu}

\end{document}